\documentclass[11pt]{amsart}

\usepackage{graphicx}
\usepackage{color}
\usepackage{pdfsync}
\usepackage[T1]{fontenc}

\usepackage[utf8]{inputenc}
\usepackage{amsfonts,amssymb,mathrsfs,cite} 
\newtheorem{thm}{Theorem}[section]

\newtheorem{lemma}[thm]{Lemma}

\theoremstyle{definition}
\newtheorem{definition}[thm]{Definition}

\theoremstyle{remark}

\numberwithin{equation}{section}

\newcommand*\diff{\mathop{}\mathrm{d}}

\newcommand\restr[2]{{
  \left.\kern-\nulldelimiterspace 
  #1 
  \vphantom{\big|} 
  \right|_{#2} 
  }}

\begin{document}

\title[]
{Martingale solution to stochastic Korteweg - de Vries equation driven by L\'evy noise}

\author[Karczewska]{Anna Karczewska}
\address{Faculty of Mathematics, Computer Science and Econometrics\\ University of Zielona G\'ora, Szafrana 4a, 65-516 Zielona G\'ora, Poland}
 \email{a.karczewska@wmie.uz.zgora.pl} \thanks{}

\author[Szczeci\'nski]{Maciej Szczeci\'nski}
\address{Faculty of Mathematics, Computer Science and Econometrics\\ University of Zielona G\'ora, Szafrana 4a, 65-516 Zielona G\'ora, Poland}
 \email{m.szczecinski@wmie.uz.zgora.pl} \thanks{}

\date{\today}




\begin{abstract}
We study stochastic Korteweg - de Vries equation driven by L\'evy noise consisting of the compensated time homogeneous Poisson random measure and a cylindrical Wiener process. We prove the existence of a martingale solution to the equation studied. In proof of the existence theorem we use the Galerkin approximation and several auxiliary results suitable for the problem considered.
\end{abstract}

\maketitle

\section{Introduction} \label{intro}

In the paper we study the stochastic Korteweg - de Vries (for short KdV) equation  with multiplicative noise of L\'evy's type
\begin{equation}\label{Levy}
\begin{cases}
\diff u(t,x) + \big( u_{3x}(t,x) + u(t,x)u_{x}(t,x) \big) \diff t = \int_{Y}F(t,u(t,x);y)\tilde{\eta}(\diff t, \diff y) + \Phi\left(t,u(t,x)\right) \diff W(t) \\
u(0,x) = u_{0}(x).
\end{cases}
\end{equation}

In the deterministic case, the assumption $u(t,x)=0$ for "large" $|x|$ leads to solitonic solutions, whereas the assumption in periodic form $u(t,x)= u(t,x+l)$ leads to periodic solutions, so-called cnoidal waves \cite{Whit,Ding}, where $l$ is the wavelength.

The deterministic Korteweg - de Vries equation \cite{KdV} (for short KdV) has been derived from the set of Eulerian shallow water and long wavelength equations. KdV can model the evolution in time, due to gravity force,  of unidirectional weakly nonlinear waves appearing at the surface of the fluid. KdV corresponds to the case of a constant pressure on the surface of the fluid and an even bottom of the container.  In more realistic physical cases small fluctuations of these quantities can be modelled by an additional random forcing term.

It is  worth to note that KdV equation became a paradigm as weakly dispersive nonlinear wave equation, since it appears naturally as first order approximation in many fields, like fluid dynamics, ion-acoustic waves in plasma, electric currents, propagation of light in fibres and many others,   see, e.g. monographs \cite{Ablow,Ding,DrazJohns,EIGR,Newell,Osborne,Remoissenet,Tao,Whit}.
Therefore it gained enormous interest among physicists, engineers, biologists and mathematicians.


The stochastic KdV equation has been studied extensively, see, e.g.\ \cite{Deb,Deb2,Deb3,Gao,Printems} and \cite{Herman,KSRB}. The mentioned above papers deal with additive and/or multiplicative noise. Some discuss exact solutions to the stochastic KdV equation. However, to the best of our knowledge, there has been no result so far for the stochastic KdV equation driven by L\'evy type noise.

In the paper we extend the results of the existence of martingale solution to that case. We apply and adapt for our purposes the approaches used in \cite{Deb,Gat} and \cite{Mot}.

\section{Existence of martingale solution to KdV} \label{exist}

Let $\left( \Omega, \mathscr{F}, \left\{\mathscr{F}_{t}\right\}_{t\geq 0}, \mathbb{P} \right)$ be a probability space with filtration and $(Y,\mathcal{Y})$ be a measurable space. 

Denote for ~$T<\infty$~ and ~$-\infty <x_1<x_2<\infty$
\begin{itemize}
\item[(i)]$\mathscr{V}$ -- the space of smooth functions $f:[0,T]\times [x_1,x_2]\rightarrow \mathbb{R}$;
\item[(ii)]$H$ -- the closure of  $\mathscr{V}$ in $L^{2}([0,T]\times [x_1,x_2]; \mathbb{R})$;
\item[(iii)]$V_{m}$, $m\geq 1$ -- the closure of $\mathscr{V}$ in $H^{m}([0,T]\times [x_1,x_2];\mathbb{R})$ [in particular $V:=V_{1}$].
\end{itemize}
Moreover, for arbitrary $m>1$ by $U$ we will denote a Hilbert space  fulfilling the following conditions
\begin{itemize}
\item[(U1)] $U\subset V_{m}$;
\item[(U2)] $U$ is dense in $V_{m}$;
\item[(U3)] embedding $U\hookrightarrow V_{m}$ is compact.
\end{itemize}

In (\ref{Levy}), $W(t)$, $t\geq 0$, is a cylindrical Wiener process adapted to filtration $\left\{\mathscr{F}_{t}\right\}_{t\geq 0}$, $\tilde{\eta}$ is a compensated time homogeneous Poisson measure on  $(Y,\mathcal{Y})$ (see definition in Appendix)
with  $\sigma$ - finite intensity measure $\nu$, $u_{0}\in H$ is a deterministic function, $u(\omega,\cdot,\cdot): \mathbb{R}_{+}\times\mathbb{R} \in \mathbb{R}$ is a c\`{a}dl\`{a}g type function for any $\omega\in\Omega$. 

A measurable function $F:[0,T]\times H\times Y \rightarrow H$  fulfils conditions
\begin{itemize}
\item[({\bf F1})] $\int_{Y} \chi_{\left\{ 0 \right\} } \left(F(t,x;y)\right)\nu (\diff y) = 0$ for all $x\in H$ and all $t\in [0,T]$;
\item[({\bf F2})] there exists a constant $L>0$, such that for all $u_{1},u_{2}\in H$ and all $t\in[0,T]$ 
\begin{equation}\label{F2}
\int_{Y} \left|F(t,u_{1};y) - F(t,u_{2};y)\right|^{2}_{H}\nu(\diff y) \leq L\left|u_{1} - u_{2}\right|^{2}_{H} \qquad \mbox{holds;}
\end{equation}
\item[({\bf F3})] there exists $\zeta >0$, such that for all $p\in\left\{1,2,2+\frac{1}{2}\zeta,4+\zeta\right\}$ there exists a constant $C_{p}>0$, such that
\begin{equation}\label{F3}
\int_{Y}\left|F(t,u;y)\right|^{p}_{H}\mu(\diff y) \leq C_{p} (1+\left|u\right|_{h}^{p}), \quad \mbox{for}\quad u\in H, ~t\in [0,T];
\end{equation}
\item[({\bf F4})] for all $v\in\mathscr{V}$ the mapping $F_{v}: L^{2}(0,T;H) \rightarrow L^{2}([0,T]\times Y),dl\otimes \nu;\mathbb{R})$, where $dl\otimes \nu$ denotes the product of Lebesgue measure and the intensity $\nu$,
defined by
\begin{equation}\label{F4}
(F_{v}(u))(t,y) := \left\langle F(t,u(t^{-});y), v \right\rangle_{H}, \quad u\in L^{2}(0,T;H)
\end{equation}
is continuous if the space
$L^{2}(0,T;H)$ is equipped with Fr{\'e}chet topology from space \\$L^{2}(0,T;H_{loc})$.

We assume that a continuous mapping $\Phi[0,T]\times V \rightarrow L_{0}^{2}(L^{2}(\mathbb{R}))$ fulfils conditions
\item[({$\boldmath{\Phi}$\bf 1})] there exists a constant $L_{\Phi}>0$, such that for all $u_{1}, u_{2}\in V$ and  all $t\in [0,T]$ 
\begin{equation}\label{P1}
\left\|\Phi(t,u_{1}) - \Phi(t,u_{2})\right\|^{2}_{L_{0}^{2}(L^{2}(\mathbb{R}))} \leq L_{\Phi} \left\|u_{1} - u_{2}\right\|^{2}_{V} \quad \mbox{holds} ;
\end{equation}
\item[({$\boldmath{\Phi}$\bf 2})] there exist constants $\alpha, \beta, \kappa>0$, such that for all  $u\in V$ and all $t\in [0,T]$ 
\begin{equation}\label{P2}
\begin{aligned}
\min & \left\{ 2 \left\langle (\mathscr{K}_{\lambda}u), u\right\rangle_{H} - \left\|\Phi(t,u)\right\|_{L_{0}^{2}(L^{2}(\mathbb{R}))}^{2}, \quad - \left\|\Phi(t,u)\right\|_{L_{0}^{2}(L^{2}(\mathbb{R}))}^{2}\right\} \\
& \geq \alpha\left|u(x)\right|_{V}^{2} - \beta \left|u(x)\right|_{H} - \kappa ,
\end{aligned}
\end{equation}
holds, where
\begin{equation}\nonumber
\begin{aligned}
(\mathscr{K}_{\lambda}u) & =  u_{3x} + \lambda\, uu_{x} ~\mbox{~for~} ~\lambda\in [0,1] ;
\end{aligned}
\end{equation}
\item[({$\boldmath{\Phi}$\bf 3})] there exists a constant $C_{\Phi}>0$, such that
\begin{equation}\label{P3}
\left\|\Phi(t,u)\right\|_{L_{0}^{2}(L^{2}(\mathbb{R}))}^{2}\le C_{\Phi} \left( \max\left\{\left|u\right|_{V}, \left|u\right|_{H}\right\} +1\right);
\end{equation}
\item[({$\boldmath{\Phi}$\bf 4})] for any $v\in\mathscr{V}$ the mapping $\Phi_{v}: L^{2}(0,T;H) \rightarrow L^{2}([0,T]; L_{0}^{2}(L^{2}(\mathbb{R})))$ given by
\begin{equation}\label{P4}
\left(\Phi_{v}(u)\right)(t) := \left\langle \Phi(t,u(t)), v \right\rangle _{H}
\end{equation}
is continuous, if space $L^{2}(0,T;H)$ is equipped with Fr{\'e}ch{\'e}t topology from $L^{2}(0,T;H_{loc})$.
\end{itemize}

\begin{definition} 
We say that the problem  (\ref{Levy}) has a {\tt  martingale solution} on the interval $[0,T]$, $T<\infty$, if there exists a basis
  $(\bar{\Omega},\bar{\mathscr{F}},\left\{\bar{\mathscr{F}}_{t}\right\}_{t\geq 0}, \bar{\mathbb{P}}, \bar{u}, \bar{\eta}, \left\{\bar{W}_{t}\right\}_{t\geq 0} )$, where
\begin{itemize}
\item[(i)] $(\bar{\Omega},\bar{\mathscr{F}},\left\{\bar{\mathscr{F}}_{t}\right\}_{t\geq 0}, \bar{\mathbb{P}})$ is a probability space with filtration;
\item[(ii)] $\bar{\eta}$ is a homogeneous Poisson random variable on measurable space  $(Y,\mathcal{Y})$ with intensity measure $\nu$;
\item[(iii)] $\left\{\bar{W}_{t}\right\}_{t\geq 0}$ is cylindrical Wiener process adapted to the filtration  $\left\{\mathscr{F}_{t}\right\}_{t\geq 0}$ ;
\item[(iv)] $\left\{u(t,x)\right\}_{t\geq 0}$ is a predictable process adapted to filtration $\left\{\mathscr{F}_{t}\right\}_{t\geq 0}$ with trajectories in 
\begin{equation}\nonumber
 \mathbb{D}(0,T;H)\cap L^{2}(0,T;V) \cap L^{\infty}(0,T;H)\cap L^{2}(0,T;H_{loc})\cap C(0,T;V'_{s,loc}(\mathbb{R}), 
\end{equation}
$s>0$, $\mathbb{P}$ - a.s., such that for all $t\in[0,T]$ and all $v\in V$  the formula 
\begin{equation}\nonumber
\begin{aligned}
& \left\langle \bar{u}(t), v \right\rangle _{H} + \int_{0}^{t} \left\langle \bar{u}_{3x}(s) + \bar{u}(s)\bar{u}_{x}(s), v\right\rangle_{H} \\
& = \left\langle \bar{u}_{0}, v\right\rangle _{H} + \int_{0}^{t}\int_{Y} \left\langle F(s,\bar{u}(s);y), v\right\rangle_{H}\bar{\eta}(\diff s, \diff y) + \left\langle \int_{0}^{t} \Phi(s,\bar{u}(s)) \diff \bar{W}(s), v \right\rangle _{H} ,
\end{aligned}
\end{equation} 
holds $\mathbb{P}$-a.s. 
\end{itemize}
\end{definition}

Now, we are able to formulate the main result of the paper.

\begin{thm}\label{T5.1}
For all  $u_{0}\in H$ and $T>0$ there exists a martingale solution to(\ref{Levy}).
\end{thm}


\begin{proof}
We construct the Galerkin approximation of the equation (\ref{Levy}).

Let $(e_{i})_{i\in\mathbb{N}}$ be an orthonormal basis in  $H$ and let $P_{m}$, $m\in\mathbb{N}$ be an orthogonal projection on $m$-dimensional space  $Sp(e_{0},...,e_{m})$. Consider initial value problem in $P_{m}H$
\begin{equation}\label{LevyG}
\begin{cases}
& \diff u^{ m}(t,x) + \Big[ u^{ m}_{3x}(t,x) + \theta\left(\frac{\left|u^{ m}_{x}(t,x)\right|}{m} \right) u^{ m}(t,x)u^{ m}_{x}(t,x) \Big]  \diff t \\
& = \int_{Y} P_{m}F(t,u(s^{-});y)\tilde{\eta}(\diff s, \diff y) + P_{m}(\Phi\left(s,u^{ m})) \diff W^{m}(t) \right.\\
& u^{ m}(0,x) = P_{m}u_{0}(x),
\end{cases}
\end{equation}
where $P_{m}L^{2}(\mathbb{R})\ni u^{ m}(t,x) = P_{m}u(t,x)$, $\sup_{m\in\mathbb{N}}\mathbb{E}\left|u^{  m}\right|^{2p} < \infty$ for all $p\geq 2$ and $\theta\in C^{\infty}(\mathbb{R})$ fulfils conditions
\begin{equation}\label{theta}
\begin{aligned}
&\theta(\xi)=1,  & \xi\in [0,\frac{m}{2}]; \\
&\theta(\xi)\in[0,1], & \xi\in [\frac{m}{2},m]; \\
&\theta(\xi)=0, & \xi > m.
\end{aligned}
\end{equation}

\begin{lemma}\label{L5.2} 
For all  $m\in\mathbb{N}$ there exists a c\`{a}dl\`{a}g process $\left\{u^{m}(t)\right\}_{t\in[0,T]}$  adapted to the filtration  $\left\{F_{t}\right\}_{t\geq 0}$ which is a martingale solution to  (\ref{LevyG}).
\end{lemma}

\begin{lemma}\label{L5.3}For all  
 $p\in[\frac{1}{2}, 2+\varsigma]$ there exist such constants $\tilde{C}_{1}(p)$, $\tilde{C}_{2}$, that
\begin{equation}\label{LevyC1}
\sup_{m\geq 1}\mathbb{E}\left(\sup_{0\leq s\leq T}\left|u^{m}(s)\right|^{2p}_{H}\right) \leq \tilde{C}_{1}(p),
\end{equation}
\begin{equation}\label{LevyC2}
\sup_{m\geq 1}\mathbb{E}\left( \int_{0}^{T} \left|u^{m}(s)\right|^{2}_{V} \diff s \right) \leq \tilde{C}_{2}.
\end{equation}
\end{lemma}

\begin{lemma}\label{L5.4} The family of distributions 
 $\mathscr{L}(u^{m})$ is tight in $\mathscr{Z}:=L^{2}(0,T;V)\cap L^{2}(0,T;H_{loc})\cap \mathbb{D}(0,T;U')\cap \mathbb{D}(0,T;H)$.
\end{lemma}

For reader's convenience proofs of Lemmas \ref{L5.2}, \ref{L5.3}, and \ref{L5.4} are given in section \ref{sc3}.

\begin{lemma}(\cite[p.~889]{Mot})\label{etaW}
Let $\eta_{m}:=\eta$ and $W_{m}:=W$, $m\in\mathbb{N}$. Then the following conditions hold
\begin{itemize}
\item[(i)] the family $\left\{\mathscr{L}(\eta_{m})\right\}_{m\in\mathbb{N}}$ is tight in $M_{\mathbb{N}}([0,t]\times Y)$;
\item[(ii)] the family $\left\{\mathscr{{L}(W_{m}})\right\}_{m\in\mathbb{N}}$  is tight in $C(0,T;\mathbb{R})$.
\end{itemize}
\end{lemma}

Due to Lemmas \ref{L5.4} and \ref{etaW} the family of distributions $\mathscr{L}(u^{m},\eta_{n},W_{n})$ is tight in $\bar{\mathscr{Z}}\times M_{\mathbb{N}}([0,t]\times Y) \times C(0,T;\mathbb{R})$. Then due to 
Corollary 7.3 in \cite{Mot}
there exists the subsequence   $\left\{u^{k}\right\}_{k\in\mathbb{N}}$, probabilistic space $\left(\bar{\Omega}, \bar{\mathscr{F}}, \left\{\bar{F}_{t}\right\}_{t\geq 0}, \bar{\mathbb{P}} \right)$ and such random variables $\left(\bar{u},\bar{\eta},\bar{W}\right)$ and $\left(\bar{u}^{k},\bar{\eta}_{k},\bar{W}_{k}\right)$, $k\in\mathbb{N}$ in this space with values in $\bar{\mathscr{Z}}$, that 
\begin{itemize}
\item[(i)] $\mathscr{L}\left(\left(\bar{u}^{k},\bar{\eta}_{k},\bar{W}_{k}\right)\right) = \mathscr{L}\left(\left(u^{m_{k}},\eta_{m_{k}},W_{m_{k}}\right)\right)$, $k\in\mathbb{N}$;
\item[(ii)] $\left(\bar{u}^{k},\bar{\eta}_{k},\bar{W}_{k}\right) \rightarrow \left(\bar{u},\bar{\eta},\bar{W}\right)$ w $\bar{\mathscr{Z}}$ a.s., when $k\rightarrow\infty$;
\item[(iii)] $\left(\bar{\eta}_{k}(\bar{\omega}),\bar{W}_{k}(\bar{\omega})\right) = \left(\eta_{m_{k}}(\bar{\omega}),W_{m_{k}}(\bar{\omega})\right)$ for all $\bar{\omega}\in\bar{\Omega}$.
\end{itemize}
Moreover, 
 $\bar{\eta}_{k}$, $k\in\mathbb{N}$ and $\bar{\eta}$ are homogeneous Poisson random measures on $(Y,\mathcal{Y})$ with intensity measure $\nu$ and $\bar{W}_{k}$, $k\in\mathbb{N}$, and $\bar{W}$ are cylindrical Wiener processes and $\bar{u}^{k} \rightarrow \bar{u}$, $\mathbb{P}$-a.s. 

Since distributions $\bar{u}^{k}$ and $u^{m_{k}}$ are identical, then due to Lemma\ref{L5.3} for all $p\in[\frac{1}{2}, 2+\varsigma]$  
\begin{equation}\label{LevyC1a}
\sup_{m\geq 1}\bar{\mathbb{E}}\left(\sup_{0\leq s\leq T}\left|\bar{u}^{m}(s)\right|^{2p}_{H}\right) \leq \tilde{C}_{1}(p)
\end{equation}
and
\begin{equation}\label{LevyC2a}
\sup_{m\geq 1}\bar{\mathbb{E}}\left( \int_{0}^{T} \left|\bar{u}^{m}(s)\right|^{2}_{V} \diff s \right) \leq \tilde{C}_{2}.
\end{equation} 

Denote
\begin{equation}\label{M}
\begin{aligned}
M^{m_{k}} (t) := & u^{m_{k}} - u_{0}^{m_{k}}  + \int_{0}^{t} \left[ u^{ m_{k}}_{3x}(s) + \theta\left(\frac{\left|u^{ m_{k}}_{x}(s)\right|}{m_{k}} \right) u^{m_{k}}(s)u^{m_{k}}_{x}(s)\right] \diff s \\
& - \int_{0}^{t} \int_{Y}  P_{m_{k}}F(t,u(s^{-},x);y)\eta_{m_{k}} (\diff s, \diff y); \\ 
\bar{M}^{k} (t) := & \bar{u}^{k} - \bar{u}_{0}^{k}  + \int_{0}^{t} \left[ \bar{u}^{ k}_{3x}(s) + \theta\left(\frac{\left|\bar{u}^{ k}_{x}(s)\right|}{k} \right) \bar{u}^{k}(s)\bar{u}^{k}_{x}(s) \right] \diff s\\
& - \int_{0}^{t} \int_{Y}  P_{k}F(t,u(s^{-},x);y)\bar{\eta}_{k}(\diff s, \diff y). \\ 
\end{aligned}
\end{equation}
Note that
\begin{equation}
M^{m_{k}} (t) = \int_{0}^{t}\left(\Phi(s, u^{m_{k}}(s))\right)\diff W^{m_{k}}(s),
\end{equation}
then it is a martingale with values in $H$, square integrable, adapted to the  filtration  $\sigma\left\{u^{m_{k}}(s), 0\leq s\leq t \right\}$ with variation 
\begin{equation}\nonumber
\left[M^{m_{k}}\right](t) := \int_{0}^{t} \Phi(s,u^{m_{k}}(s))\left[\Phi(s,u^{m_{k}}(s))\right]^{*} \diff s .
\end{equation}

Substitute in the Doob inequality (e.g., see Theorem~2.2 in \cite{Gaw})  $M_{t} := M^{m_{k}} (t) $ and $p:=2p$.  Then there exists  $K'_{p}$, such that
\begin{equation}\label{szacDoobL}
\mathbb{E}\left[\left(\sup_{t\in[0,T]}\left|M^{m_{k}} (t)\right|^{p}_{H}\right)\right] \leq K_{p}' .
\end{equation}
Let $0\leq s \leq t \leq T$ and let $\varphi$ be a bounded continuous function on $L^{2}(0,s;H_{loc})$ and $a\in H$ be an arbitrary and fixed. Since $M^{m_{k}}$ is a martingale and $\mathscr{L}(\bar{u}^{m_{k}}) = \mathscr{L}(u^{m_{k}})$, then
\begin{equation}\nonumber
\begin{aligned}
\mathbb{E}\left( \left\langle  M^{m_{k}(t)} - M^{m_{k}}(s); a \right\rangle_{\!H} \varphi(u^{m_{k}}(s))\right) = 0, \\
\mathbb{E}\left( \left\langle  \bar{M}^{m_{k}(t)} - \bar{M}^{m_{k}}(s); a \right\rangle_{\!H} \varphi(\bar{u}^{m_{k}}(s))\right) = 0 .
\end{aligned}
\end{equation}

Denote
\begin{equation}\nonumber
\begin{aligned}
\bar{M} (t) :=  \bar{u} - u_{0}  + \int_{0}^{t} \left[  \bar{u}_{3x}(s) +  \bar{u}(s)\bar{u}_{x}(s) \right]  \diff s - \int_{0}^{t} \int_{Y} F(t,\bar{u}(s^{-},x);y)\bar{\eta} (\diff s, \diff y) . 
\end{aligned}
\end{equation}
We will show that
 $P_{m_{k}}\int_{0}^{t} \int_{Y} F(s,\bar{u}^{m_{k}}(s^{-},x);y)\bar{\eta}_{m_{k}} (\diff s, \diff y) \rightarrow \int_{0}^{t} \int_{Y} F(s,\bar{u}(s^{-},x);y)\bar{\eta} (\diff s, \diff y)$, when $m_{k}\rightarrow\infty$. Let  $\nu\in U$ be arbitrary fixed. For all $t\in[0,T]$ we have
\begin{equation}\nonumber
\begin{aligned}
&\int_{0}^{t}\int_{Y}\left|\left\langle \left[ F(s,\bar{u}^{m_{k}}(s^{-});y) - F(s,\bar{u}(s^{-});y)\right]; \nu \right\rangle_H\right|^{2}\diff \nu(y) \diff s \\
= & \int_{0}^{t}\int_{Y} \left| F_{\nu}(\bar{u}^{m_{k}})(s,y) - F_{\nu}(\bar{u})(s,y) \right|^{2} \diff \nu(y)\diff s \\
\leq & \int_{0}^{T}\int_{Y} \left| F_{\nu}(\bar{u}^{m_{k}})(s,y) - F_{\nu}(\bar{u})(s,y) \right|^{2} \diff \nu(y)\diff s \\
= & \left| F_{\nu}(\bar{u}^{m_{k}})(s,y) - F_{\nu}(\bar{u})(s,y) \right|^{2}_{L^{2}([0,T]\times Y;\mathbb{R})} .
\end{aligned}
\end{equation}
In above equation  $F_{\nu}$ is the function defined by (\ref{F4}). Due to condition  (\ref{F2}) there is
\begin{equation}\nonumber
\int_{0}^{t}\int_{Y}\left|\left\langle\left[ F(s,\bar{u}^{m_{k}}(s^{-});y) -   F(s,\bar{u}(s^{-});y)\right]; \nu \right\rangle_{H}\right|^{2} \diff\nu(y)\diff s \leq   \int_{0}^{t} \left| \left\langle\left[ \bar{u}^{m_{k}}(s^{-}) - \bar{u}(s^{-})\right]; \nu \right\rangle_{H} \right|^{2} \diff s
\end{equation}
and since  $\bar{u}^{m_{k}}\rightarrow\bar{u}$ when $m_{k}\rightarrow\infty$,
\begin{equation}\label{Flim}
\int_{0}^{t}\int_{Y}\left|\left\langle\left[ F(s,\bar{u}^{m_{k}}(s^{-});y)- F(s,\bar{u}(s^{-});y)\right]; \nu \right\rangle _{H}\right|^{2} \diff\nu(y)\diff s \rightarrow 0, \quad m_{k}\rightarrow\infty .
\end{equation}
Moreover, by inequalities
 (\ref{F3}) and (\ref{LevyC1}), for arbitrary fixed $t\in[0,T]$, $r\in\left.\left(1,2+\frac{\varsigma}{2}\right]\right.$, $n\in\mathbb{N}$, there exist constants $C_{1}(r),C_{2},C_{3},C_{4}(r)>0$, such that
\begin{equation}\label{FC}
\begin{aligned}
&\bar{\mathbb{E}}\left[\left|\int_{0}^{t}\int_{Y}\left|\left\langle\left[ F(s,\bar{u}^{m_{k}}(s^{-});y) - F(s,\bar{u}(s^{-});y)\right]; \nu \right\rangle _{H}\right|^{2}\diff \nu(y)\diff s\right|^{r}\right] \\
\leq & 2^{r}\left|\nu\right|^{2r}_{H}\bar{\mathbb{E}}\left[\left|\int_{0}^{t}\int_{Y}\left\{\left|F(s,\bar{u}^{m_{k}}(s^{-});y)\right|_{H}^{2} + \left|F(s,\bar{u}(s^{-});y)\right|_{H}^{2}\right\}\diff\nu(y)\diff s\right|^{r}\right] \\
\leq & 2^{r}C_{2}^{r}\left|\nu\right|^{2r}_{H}\bar{\mathbb{E}}\left[\left|\int_{0}^{t}\left\{2+\left|\bar{u}^{m_{k}}(s)\right|_{H}^{2}+\left|\bar{u}(s)\right|^{2}_{H}\right\}\diff s\right|^{r}\right] \\
\leq & C_{3}\left(1+\bar{\mathbb{E}}\left[\sup_{s\in[0,T]}\left|\bar{u}^{m_{k}}(s)\right|^{2r}_{H}\right]\right) \leq c(1+C_{1}(r)) \leq C_{4}(r) .
\end{aligned}
\end{equation}
Due to inequalities (\ref{Flim}) and (\ref{FC}) we have
\begin{equation}\label{EFlim}
\bar{\mathbb{E}}\left[\int_{0}^{t}\int_{Y}\left|\left\langle\left[ F(s,\bar{u}^{m_{k}}(s^{-});y) - F(s,\bar{u}(s^{-});y)\right];\nu\right\rangle_{H}\right|^{2}\diff \nu(y) \diff s\right] \rightarrow 0 ~\mbox{~as~}~ 
 m_{k}\rightarrow\infty .
\end{equation}
Now, take arbitrary fixed  $\tilde{\nu}\in H$ and $\varepsilon > 0$. Since
 $\mathscr{V}$ is tight in $H$, then there exists  $\nu_{\varepsilon}\in\mathscr{V}$, such that $\left|\tilde{\nu} - \nu_{\varepsilon}\right| \leq \varepsilon$. By (\ref{F3}) there exists a constant $C_{5}>0$, such that
\begin{equation}\nonumber
\begin{aligned}
& \int_{0}^{t}\int_{Y}\left|\left\langle\left[ F(s,\bar{u}^{m_{k}}(s^{-});y) - F(s,\bar{u}(s^{-});y)\right]; \tilde{\nu}\right\rangle_{H}\right|^{2}\diff \tilde{\nu}(y)\diff s \\
\leq & 2 \int_{0}^{t}\int_{Y}\left|\left\langle\left[ F(s,\bar{u}^{m_{k}}(s^{-});y) - F(s,\bar{u}(s^{-});y)\right]; (\tilde{\nu}-\nu_{\varepsilon})\right\rangle_{H}\right|^{2}\diff \tilde{\nu}(y)\diff s \\
& + 2 \int_{0}^{t}\int_{Y}\left|\left\langle\left[ F(s,\bar{u}^{m_{k}}(s^{-});y) - F(s,\bar{u}(s^{-});y)\right]; \nu_{\varepsilon}\right\rangle_{H}\right|^{2}\diff \tilde{\nu}(y)\diff s \\
\leq & 4C_{5}\varepsilon^{2} \int_{0}^{t} \left\{2+\left|\bar{u}^{m_{k}}(s)\right|^{2}_{H} + \left|\bar{u}\right|^{2}_{H}\right\}\diff s \\
& + 2 \int_{0}^{t}\int_{Y}\left|\left\langle\left[ F(s,\bar{u}^{m_{k}}(s^{-});y) - F(s,\bar{u}(s^{-});y)\right]; \nu_{\varepsilon}\right\rangle_{H}\right|^{2}\diff \tilde{\nu}(y)\diff s ,
\end{aligned}
\end{equation}
so, due to (\ref{LevyC1}) there exist constants $C_{6},C_{7}>0$, such that
\begin{equation}\label{pomLevy}
\begin{aligned}
& \bar{\mathbb{E}}\left[\int_{0}^{t}\int_{Y}\left|\left\langle\left[ F(s,\bar{u}^{m_{k}}(s^{-});y) - F(s,\bar{u}(s^{-});y)\right]; \tilde{\nu}\right\rangle_{H}\right|^{2}\diff \tilde{\nu}(y)\diff s\right] \\
\leq & 4C_{5}\varepsilon^{2} \bar{\mathbb{E}}\left[\int_{0}^{t} \left\{2+\left|\bar{u}^{m_{k}}(s)\right|^{2}_{H} + \left|\bar{u}\right|^{2}_{H}\right\}\diff s\right] \\
& + 2 \bar{\mathbb{E}}\left[ \int_{0}^{t}\int_{Y}\left|\left\langle\left[ F(s,\bar{u}^{m_{k}}(s^{-});y) - F(s,\bar{u}(s^{-});y)\right]; \nu_{\varepsilon}\right\rangle_{H}\right|^{2}\diff \tilde{\nu}(y)\diff s\right] \\
\leq & C_{6}\varepsilon^{2} + 2\bar{\mathbb{E}}\left[ \int_{0}^{t}\int_{Y}\left|\left\langle\left[ F(s,\bar{u}^{m_{k}}(s^{-});y) - F(s,\bar{u}(s^{-});y)\right]; \nu_{\varepsilon}\right\rangle_{H}\right|^{2}\diff \tilde{\nu}(y)\diff s\right].
\end{aligned}
\end{equation}
Taking in (\ref{pomLevy}) $m_k\to\infty$ and using (\ref{EFlim}) one obtains
\begin{equation}\nonumber
\limsup_{m_{k}\rightarrow\infty}\bar{\mathbb{E}}\left[\int_{0}^{t}\int_{Y}\left|\left\langle\left[ F(s,\bar{u}^{m_{k}}(s^{-});y) - F(s,\bar{u}(s^{-});y)\right]; \tilde{\nu}\right\rangle_{H}\right|^{2}\diff \tilde{\nu}(y)\diff s\right] \leq C_{6}\varepsilon^{2} .
\end{equation}
Since  $\varepsilon>0$ was arbitrary, then
\begin{equation}\nonumber
\limsup_{m_{k}\rightarrow\infty}\bar{\mathbb{E}}\left[\int_{0}^{t}\int_{Y}\left|\left\langle\left[ F(s,\bar{u}^{m_{k}}(s^{-});y) - F(s,\bar{u}(s^{-});y)\right]; \tilde{\nu}\right\rangle_{H}\right|^{2}\diff \tilde{\nu}(y)\diff s\right] = 0 ,
\end{equation}
so
\begin{equation}\nonumber
\bar{\mathbb{E}}\left[\int_{0}^{t}\int_{Y}\left|\left\langle\left[ F(s,\bar{u}^{m_{k}}(s^{-});y) - F(s,\bar{u}(s^{-});y)\right]; \tilde{\nu}\right\rangle_{H}\right|^{2}\diff \tilde{\nu}(y)\diff s\right] \rightarrow 0 ~\mbox{~as~}~ 
m_{k}\rightarrow\infty
\end{equation}
and since  $\bar{\eta}_{m_{k}}=\bar{\eta}$,
\begin{equation}\label{limPnF}
\bar{\mathbb{E}}\left[\left|\int_{0}^{t}\int_{Y}\left|\left\langle\left[ P_{m_{k}}F(s,\bar{u}^{m_{k}}(s^{-});y) - F(s,\bar{u}(s^{-});y)\right]; \tilde{\nu}\right\rangle_{H}\right|^{2}\diff \tilde{\bar{\eta}}(\diff s, \diff y)\right|^{2}\right] \rightarrow 0 ~\mbox{~as~}~ 
 m_{k}\rightarrow\infty,
\end{equation}
where $\tilde{\bar{\eta}}$ denotes compensated Poisson random measure  corresponding to $\bar{\eta}$. Using (\ref{F3}) and (\ref{LevyC1}) one obtains 
\begin{equation}\label{szacPnF}
\begin{aligned}
& \bar{\mathbb{E}}\left[\left|\int_{0}^{t}\int_{Y}\left|\left\langle\left[ P_{m_{k}}F(s,\bar{u}^{m_{k}}(s^{-});y) - F(s,\bar{u}(s^{-});y)\right]; \tilde{\nu}\right\rangle_{H}\right|^{2}\diff \tilde{\bar{\eta}}(\diff s, \diff y)\right|^{2}\right] \\
= & \bar{\mathbb{E}}\left[\left|\int_{0}^{t}\int_{Y}\left|\left\langle\left[ P_{m_{k}}F(s,\bar{u}^{m_{k}}(s^{-});y) - F(s,\bar{u}(s^{-});y)\right]; \tilde{\nu}\right\rangle_{H}\right|^{2}\diff \nu(\diff y) \diff s \right|^{2}\right] \\
\leq & 2C_{7}\left|\nu\right|_{H}^{2}\bar{\mathbb{E}}\left[ \int_{0}^{t}\left\{2+\left|\bar{u}^{m_{k}}(s)\right|^{2}_{H} + \left|\bar{u}(s)\right|^{2}_{H}\right\}\diff s \right] \\
\leq & C_{8}\left(1+\bar{\mathbb{E}}\left[\sup_{s\in[0,T]}\left|\bar{u}^{m_{k}}\right|^{2}_{H}\right]\right) \leq C_{9} (1+C_{10}(2))\leq C_{11}
\end{aligned}
\end{equation}
for some $C_{7},C_{8},C_{9},C_{10},C_{11}>0$. Due to (\ref{limPnF}) and (\ref{szacPnF}) we have for all $\nu\in H$
\begin{equation}\label{IMP1}
\int_{0}^{T}\bar{\mathbb{E}}\left[\left|\int_{0}^{t}\int_{Y}\left\langle\left[ P_{m_{k}}F(s,\bar{u}^{m_{k}}(s^{-});y) - F(s,\bar{u}(s^{-});y)\right]; \nu\right\rangle_{H}\tilde{\bar{\eta}}(\diff s, \diff y)\right|^{2}\right]\diff t \rightarrow 0 \mbox{~when~} 
m_{k}\rightarrow\infty.
\end{equation}
This is true for all $\nu\in U$, as well (since $u\in H$).

If $k_{m}\rightarrow\infty$ then $\bar{M}^{k_{m}}(t) \rightarrow \bar{M}(t)$ and $\bar{M}^{k_{m}}(s) \rightarrow \bar{M}(s)$, $\bar{\mathbb{P}}$-a.s.\ and since  $\varphi$ is continuous, $\varphi(\bar{u}^{m_{k}}(s)) \rightarrow \varphi(\bar{u}(s))$, $\mathbb{P}$-a.s.  This and (\ref{IMP1}) gives for $k_{m}\rightarrow\infty$ 
\begin{equation}\nonumber
\mathbb{E}\left( \left\langle\left[ \bar{M}^{m_{k}}(t) - \bar{M}^{m_{k}}(s)\right]; a \right\rangle_{H} \varphi(\bar{u}^{m_{k}}(s))\right)  \rightarrow \mathbb{E}\left( \left\langle\left[  \bar{M}(t) - \bar{M}(s)\right]; a \right\rangle_{H} \varphi(\bar{u}(s))\right). 
\end{equation}
In particular 
\begin{equation}\label{Mpom}
\mathbb{E}\left( \left\langle  \bar{M}^{m_{k}}(t); a \right\rangle_{H} \varphi(\bar{u}^{m_{k}}(s))\right)  \rightarrow \mathbb{E}\left( \left\langle \bar{M}(t); a \right\rangle_{H} \varphi(\bar{u}(s))\right).
\end{equation}

Moreover, from (\ref{IMP1}) and (\ref{Mpom}) for all $a\in U$
\begin{equation}\label{IMP2}
\begin{aligned}
& \int_{0}^{t}\left\langle\left[ \bar{u}^{m_{k}}_{3x}(s) + \bar{u}^{m_{k}}(s)\bar{u}^{m_{k}}_{x}(s)\right] ; a\right\rangle_{H}\diff s \rightarrow & \int_{0}^{t}\left\langle\left[ \bar{u}_{3x}(s) + \bar{u}(s)\bar{u}_{x}(s)\right] ; a\right\rangle_{H}\diff s .
\end{aligned}
\end{equation}
By  \cite[p.~895-898]{Mot},  for all $a\in U$  
\begin{equation}\label{IMP3}
\int_{0}^{T}\bar{\mathbb{E}}\left[\left|\left\langle \int_{0}^{t}\left[P_{m_{k}}\Phi(s,\bar{u}^{m_{k}}(s)) - \Phi(s,\bar{u}(s))\right]\diff\bar{W}(s);a\right\rangle_{H}\right|^{2}\right] \diff t \to 0 ~\mbox{~as~}~ m_{k}\to\infty .
\end{equation}

Let $a\in U$. Denote
\begin{equation}\nonumber
\begin{aligned}
\bar{N}^{m_{k}}(t) := & \left\langle \bar{u}^{m_{k}}(0) ; a \right\rangle_{H} + \int_{0}^{t} \left\langle\left[ \bar{u}^{m_{k}}_{3x}(s) + \bar{u}^{m_{k}}(s)\bar{u}^{m_{k}}_{x}(s)\right] ; a\right\rangle_{H}\diff s \\
& + \int_{0}^{t} \int_{Y} \left\langle  P_{m_{k}}F(s,\bar{u}^{m_{k}}(s^{-});y); a\right\rangle_{H} \bar{\eta}_{m_{k}}(\diff s, \diff y) \\
& + \left\langle \int_{0}^{t} 
 P_{m_{k}}\Phi(s,\bar{u}^{m_{k}}(s)) \diff \bar{W}_{m_{k}}(s); a\right\rangle_{H} ;  \\
N^{m_{k}}(t) := & \left\langle u^{m_{k}}(0) ; a \right\rangle_{H} + \int_{0}^{t} \left\langle\left[ u^{m_{k}}_{3x}(s) + u^{m_{k}}(s)u^{m_{k}}_{x}(s)\right] ; a\right\rangle_{H}\diff s \\
& + \int_{0}^{t} \int_{Y} \left\langle  P_{m_{k}}F(s,u^{m_{k}}(s^{-});y); a\right\rangle_{H} \eta_{m_{k}}(\diff s, \diff y) \\
& + \left\langle \int_{0}^{t} 
P_{m_{k}}\Phi(s,u^{m_{k}}(s)) \diff W_{m_{k}}(s); a\right\rangle_{H} ;  \\
\bar{N}(t) := & \left\langle \bar{u}(0) ; a \right\rangle_{H} + \int_{0}^{t} \left\langle\left[ \bar{u}_{3x}(s) + \bar{u}(s)\bar{u}_{x}(s)\right] ; a\right\rangle_{H}\diff s \\
& + \int_{0}^{t} \int_{Y} \left\langle F(s,\bar{u}(s^{-});y); a\right\rangle_{H} \bar{\eta}(\diff s, \diff y) + \left\langle \int_{0}^{t} 
\Phi(s,\bar{u}(s)) \diff \bar{W}(s); a\right\rangle_{H}  .
\end{aligned}
\end{equation}
Since $u^{m_K}$ is the solution of equation (\ref{LevyG}) for $m:=m_{k}$, then  for all $t\in[0,T]$ and $a\in U$
\begin{equation}\nonumber
\left\langle u^{m_{k}}(t); a \right\rangle_{H} = N^{m_{k}}(t), \quad \mathbb{P}-a.s.
\end{equation}
In particular
\begin{equation}\nonumber
\int_{0}^{T}\mathbb{E}\left[\left|\left\langle u^{m_{k}}(t); a \right\rangle_{H} - N^{m_{k}}(t)\right|^{2}\right] \diff t = 0.
\end{equation}
Because $\mathscr{L}(u^{m_{k}},\eta_{m_{k}},W_{m_{k}}) = \mathscr{L}(u,\eta,W)$, so
\begin{equation}\nonumber
\int_{0}^{T}\bar{\mathbb{E}}\left[\left|\left\langle \bar{u}^{m_{k}}(t); a \right\rangle_{H} - \bar{N}^{m_{k}}(t)\right|^{2}\right] \diff t = 0
\end{equation}
and
\begin{equation}\nonumber
\int_{0}^{T}\bar{\mathbb{E}}\left[\left|\left\langle \bar{u}(t); a \right\rangle_{H} - \bar{N}(t)\right|^{2}\right] \diff t = 0.
\end{equation}
This implies
\begin{equation}\nonumber
\left\langle \bar{u}(t); a \right\rangle_{H} = \bar{N}(t), \quad \bar{\mathbb{P}}-a.s.,
\end{equation}
and also
\begin{equation}\label{fin1}
\begin{aligned}
&\left\langle \bar{u}^{m_{k}}(t); a \right\rangle_{H} - \left\langle \bar{u}^{m_{k}}(0); a \right\rangle_{H} + \int_{0}^{t} \left\langle\left[ u^{m_{k}}_{3x}(s) + u^{m_{k}}(s)u^{m_{k}}_{x}(s)\right] ; a\right\rangle_{H}\diff s \\
& - \int_{0}^{t} \int_{Y} \left\langle F(s,\bar{u}(s^{-});y); a\right\rangle_{H} \bar{\eta}(\diff s, \diff y) - \left\langle \int_{0}^{t} 
\Phi(s,\bar{u}(s)) \diff \bar{W}(s); a\right\rangle_{H}  = 0,
\end{aligned}
\end{equation}
$\bar{\mathbb{P}}$-a.s.\ on $\bar{\Omega}$ and $l$-a.s.\ on $[0,T]$, where $l$ is the Lebesgue measure. Since $\bar{u}$ has values in $\mathscr{Z}$, in particular $\bar{u}\in\mathbb{D}(0,T;H)$, then the function on l.h.s.\ of the inequality 
 (\ref{fin1}) is c\`{a}dl\`{a}g type with respect to $t$. Because two  c\`{a}dl\`{a}g type functions equal for almost all 
 $t\in[0,T]$ have to be equal for all
 $t\in[0,T]$, so for all $t\in[0,t]$ and all $a\in U$
\begin{equation}\label{fin2}
\begin{aligned}
&\left\langle \bar{u}^{m_{k}}(t); a \right\rangle_{H} - \left\langle \bar{u}^{m_{k}}(0); a \right\rangle_{H} + \int_{0}^{t} \left\langle\left[ u^{m_{k}}_{3x}(s) + u^{m_{k}}(s)u^{m_{k}}_{x}(s)\right] ; a\right\rangle_{H}\diff s \\
& - \int_{0}^{t} \int_{Y} \left\langle F(s,\bar{u}(s^{-});y); a\right\rangle_{H} \bar{\eta}(\diff s, \diff y) - \left\langle \int_{0}^{t} 
 \Phi(s,\bar{u}(s)) \diff \bar{W}(s); a\right\rangle_{H}  = 0, \quad \bar{\mathbb{P}}\mbox{-a.s.}
\end{aligned}
\end{equation}
Moreover, since 
 $U$ is dense in $V$, then inequality (\ref{fin2}) holds for all $a\in V$. Then $\bar{u}$ is the required martingale solution of the problem (\ref{Levy}). What finishes the proof of Theorem \ref{T5.1}.

\end{proof}


\section{Proofs of Lemmas \ref{L5.2}, \ref{L5.3} and \ref{L5.4}} \label{sc3}
We start with the following auxiliary result.

\begin{proof}[Proof of Lemma \ref{L5.2}]
\begin{lemma}(\cite{IW}) Begin with initial value problem
\begin{equation}\label{IWLevy}
\begin{cases}
\diff u(t) = \sigma(u(t)) \diff W(t) + b(t,u(t)) \diff t + \int_{Y} F(t,u(t^{-});y)\tilde{\eta}(\diff t, \diff y) \\
u(0)=u_{0},
\end{cases}
\end{equation}
where $\sigma(u(t))$ and $b(t,u(t))$ are continuous, $\tilde{\eta}$ is a homogeneous compensated Poisson random measure on  $(Y,\mathcal{Y})$ with $\sigma$-finite intensity measure $\nu$ and $\mathbb{E}u_{0} < \infty$. Let $\sigma (u)$, $b(u)$ and $F(t,u,y)$ fulfil the condition
\begin{equation}\label{Levycond}
\left|\sigma(u) - \sigma(v)\right|^{2} + \int_{Y}\left|F(t,u,y)-F(t,v,y)\right|^{2}\tilde{\eta}(\diff t, \diff y) + \left\|b(u)-b(v)\right\|^{2} \leq K\left|u-v\right|^{2}
\end{equation}
for some $K>0$ and arbitrary $u,v$. Then  (\ref{IWLevy}) has a martingale solution.
\end{lemma}
Let $m\in\mathbb{N}$, $u,v\in H$ be arbitrary fixed and 
\begin{equation}\nonumber
\begin{aligned}
b(u(t)) & := \theta\left(\frac{u_{x}^{m}(t)}{m}\right)u^{m}(t)u^{m}_{x}(t) ; \\
\sigma(t,u(t)) & := P_{m}\Phi(t,u^{m}(t)); \\
F(t,u(t),y) & := P_{m} F(t,u^{m}(t);y) .
\end{aligned}
\end{equation}
We have
\begin{equation}\label{Gal1}
\begin{aligned}
& \left|\theta\left(\frac{u_{x}^{m}(t)}{m}\right)u^{m}(t)u^{m}_{x}(t) - \theta\left(\frac{v_{x}^{m}(t)}{m}\right)v^{m}(t)v^{m}_{x}(t) \right|_{H}^{2}  \\
= & \left|\theta\left(\frac{u_{x}^{m}(t)}{m}\right)u^{m}(t)u^{m}_{x}(t) - \theta\left(\frac{v_{x}^{m}(t)}{m}\right)v^{m}(t)v^{m}_{x}(t) \right|_{L^{2}(\mathbb{R})}^{2}  \\
= & \int_{\mathbb{R}} \left[\theta\left(\frac{u_{x}^{m}(t)}{m}\right)u^{m}(t)u^{m}_{x}(t) - \theta\left(\frac{v_{x}^{m}(t)}{m}\right)v^{m}(t)v^{m}_{x}(t)\right]^{2} \diff x \\
= & \int_{\mathbb{R}} \left[ \theta\left(\frac{u_{x}^{m}(t)}{m}\right)u^{m}(t)u^{m}_{x}(t) \right]^{2} \diff x + \int_{\mathbb{R}} \left[ \theta\left(\frac{v_{x}^{m}(t)}{m}\right)v^{m}(t)v^{m}_{x}(t) \right]^{2} \diff x \\
& - \int_{\mathbb{R}} 2 \left[ \theta\left(\frac{u_{x}^{m}(t)}{m}\right)u^{m}(t)u^{m}_{x}(t)\right] \left[\theta\left(\frac{v_{x}^{m}(t)}{m}\right)v^{m}(t)v^{m}_{x}(t)\right] \diff x \\
\leq & m^{4} \int_{\mathbb{R}} \left[ u^{m}(t) \right]^{2} \diff x + m^{4} \int_{\mathbb{R}} \left[ v^{m}(t) \right]^{2} \diff x - m^{4} \int_{\mathbb{R}} 2 u^{m}(t)v^{m}(t) \diff x \\
= & m^{4} \int_{\mathbb{R}} \left[ u^{m}(t) - v^{m}(t)\right]^{2} \diff x = m^{4} \left|u^{m}(t) - v^{m}(t)\right|_{L^{2}(\mathbb{R})} = m^{4} \left|u^{m}(t) - v^{m}(t)\right|_{H} .
\end{aligned}
\end{equation}
Moreover, due to conditions  (\ref{F2}) and (\ref{P1}), there exist constants $L_{F},L_{\Phi}, C_{1}>0$, such that
\begin{equation}\label{Gal2}
\begin{aligned}
& \int_{Y}\left|F(t,u^{m}(t);y) - F(t,v^{m}(t);y)\right|^{2}_{H} \tilde{\eta}(\diff t, \diff y) + \left\|\Phi(t,u^{m}(t)) - \Phi(t,v^{m}(t))\right\|^{2}_{L_{0}^{2}(H)} \\
\leq &  L_{F} \left| u^{m}(t) - v^{m}(t) \right|_{H} + L_{\Phi} \left| u^{m}(t) - v^{m}(t) \right|_{V} \leq L_{F} \left| u^{m}(t) - v^{m}(t) \right|_{H} + L_{\Phi}C_{1} \left| u^{m}(t) - v^{m}(t) \right|_{H} .
\end{aligned}
\end{equation}
Addition inequalities  (\ref{Gal1}) and (\ref{Gal2}) yield
\begin{equation}\nonumber
\begin{aligned}
& \left|\theta\left(\frac{u_{x}^{m}(t)}{m}\right)u^{m}(t)u^{m}_{x}(t) - \theta\left(\frac{v_{x}^{m}(t)}{m}\right)v^{m}(t)v^{m}_{x}(t) \right|_{H}^{2} \\
&+ \int_{Y}\left|F(t,u^{m}(t);y) - F(t,v^{m}(t);y)\right|^{2}_{H} \tilde{\eta}(\diff t, \diff y) 
 + \left\|\Phi(t,u^{m}(t)) - \Phi(t,v^{m}(t))\right\|^{2}_{L_{0}^{2}(H)} \\
& \leq m^{4} \left|u^{m}(t) - v^{m}(t)\right|_{H} + L_{F} \left| u^{m}(t) - v^{m}(t) \right|_{H} + L_{\Phi}C_{1} \left| u^{m}(t) - v^{m}(t) \right|_{H} \\
& \leq \max\left\{ m^{4}, L_{F}, L_{\Phi}C_{1} \right\} \left|u^{m}(t) - v^{m}(t)\right|_{H} .
\end{aligned}
\end{equation}
This finishes the proof since 
 it is enough to substitute $K:=\max\left\{ m^{4}, L_{F}, L_{\Phi}C_{1} \right\}$ in (\ref{IWLevy}).
\end{proof}

\begin{proof}[Proof of the Lemma \ref{L5.3}]
Denote
\begin{equation}\label{stop}
\tau_{m} := \inf\left\{t\geq 0: \quad \left|u^{m}(t)\right|_{H} \geq R\right\}, \quad m\in\mathbb{N}, R>0 .
\end{equation}
Since every  process $\left\{u^{m}(t)\right\}_{t\in[0,T]}$ is $\mathscr{F}_{t}$-adapted and right-continuous, then $\tau_{m}(R)$ is its stopping time. 
Moreover, since  $\left\{u^{m}(t)\right\}_{t\in[0,T]}$ is c\`{a}dl\`{a}g type, so its trajektories $t\mapsto u^{m}(t)$ are bounded on $[0,T]$, $\mathbb{P}$-a.s.\ and $\tau_{m}\uparrow T$, $\mathbb{P}$-a.s.\ when $R\uparrow\infty$. 

Let $p=1$ or $p=2+\frac{\varsigma}{2}$ and let $\theta:=\theta\left(\frac{\left|u_{x}^{m}(t)\right|}{m}\right)\leq 1$. Applying the It{\^o} formula to function $A(u^{m}(t)) := \left|u^{m}(t)\right|_{H}^{2p}$ one obtains, similarly like in the proof of Lemma~2.4 in \cite{KaSz17}
\begin{equation}\nonumber
\begin{aligned}
& \left|u^{ m}(t\wedge\tau_{m}(R))\right|_{H}^{2p} =  \left|u^{ m}(t\wedge\tau_{m}(R))\right|_{L^{2}(\mathbb{R})}^{2p} \\
=& \left|P_{n}u_{0}\right|^{2p}_{H} +  \int_{0}^{t\wedge\tau_{m}(R)} 2p\left|u^{ m}(s)\right|_{L^{2}(\mathbb{R})}^{2p-2}\left\langle u^{ m}(s) ; \Phi^{m} \big(u^{ m}(s)\big) \diff W^{m}(s) \right\rangle_{H} \\
& - \int_{0}^{t\wedge\tau_{m}(R)} 2p \left|u^{ m}(s)\right|_{L^{2}(\mathbb{R})}^{2p-2} \left\langle u^{ m}(s) ; \mathscr{K}_{\theta}\big( u^{ m}(s) \big) \right\rangle_{H} \diff s \\
& + \int_{0}^{t\wedge\tau_{m}(R)} p^{2} \left|u^{ m}(s)\right|_{L^{2}(\mathbb{R})}^{2p-2} \left| \Phi^{m} \big(u^{ m}(s)\big) \right|^{2}_{L^{2}_{0}(L^{2}(\mathbb{R}))}  \\
& + \int_{0}^{t\wedge\tau_{m}(R)} \int_{Y} \left[ \left|u^{m}(s^{-})+P_{m}F(s,u^{m}(s^{-});y)\right|^{2}_{L^{2}_{0}(L^{2}(\mathbb{R}))} - \left|u^{m}(s^{-})\right|^{2}_{L^{2}_{0}(L^{2}(\mathbb{R}))} \right] \tilde{\eta}(\diff s, \diff y) \\
& + \int_{0}^{t\wedge\tau_{m}(R)} \int_{Y} \left[ \left|u^{m}(s^{-})+P_{m}F(s,u^{m}(s^{-});y)\right|^{2}_{L^{2}_{0}(L^{2}(\mathbb{R}))} - \left|u^{m}(s^{-})\right|^{2}_{L^{2}_{0}(L^{2}(\mathbb{R}))} \right. \\
& \left. - 2p \left|u^{m}(s^{-})\right|^{2p-2}_{L^{2}_{0}(L^{2}(\mathbb{R}))} \left\langle u^{m}(s^{-}); P_{m}F(s,u^{m}(s^{-});y)\right\rangle_{H}\right] \nu (\diff y) \diff s .
\end{aligned}
\end{equation}
Denote
\begin{equation}\nonumber
\begin{aligned}
K^{m}(t) := & \int_{0}^{t\wedge\tau_{m}(R)} 2p\left|u^{ m}(s)\right|_{L^{2}(\mathbb{R})}^{2p-2}\left\langle u^{ m}(s) ; \Phi^{m} \big(u^{ m}(s)\big) \diff W^{m}(s) \right\rangle_{H} \\
& - \int_{0}^{t\wedge\tau_{m}(R)} 2p \left|u^{ m}(s)\right|_{L^{2}(\mathbb{R})}^{2p-2} \left\langle u^{ m}(s); \mathscr{K}_{\theta}\big( u^{ m}(s) \big) \right\rangle_{H} \diff s \\
& + \int_{0}^{t\wedge\tau_{m}(R)} p^{2} \left|u^{ m}(s)\right|_{L^{2}(\mathbb{R})}^{2p-2} \left| \Phi^{m} \big(u^{ m}(s)\big) \right|^{2}_{L^{2}_{0}(L^{2}(\mathbb{R}))} \diff s ; \\
M^{m}(t) := & \int_{0}^{t\wedge\tau_{m}(R)} \int_{Y} \left[ \left|u^{m}(s^{-})+P_{m}F(s,u^{m}(s^{-});y)\right|^{2p}_{L^{2}_{0}(L^{2}(\mathbb{R}))} - \left|u^{m}(s^{-})\right|^{2p}_{L^{2}_{0}(L^{2}(\mathbb{R}))} \right] \tilde{\eta}(\diff s, \diff y); \\
I^{m}(t) := & \int_{0}^{t\wedge\tau_{m}(R)} \int_{Y} \left[ \left|u^{m}(s^{-})+P_{m}F(s,u^{m}(s^{-});y)\right|^{2p}_{L^{2}_{0}(L^{2}(\mathbb{R}))} - \left|u^{m}(s^{-})\right|^{2p}_{L^{2}_{0}(L^{2}(\mathbb{R}))} \right. \\
& \left. - 2p \left|u^{m}(s^{-})\right|^{2p-2}_{L^{2}_{0}(L^{2}(\mathbb{R}))} \left\langle u^{m}(s^{-}), P_{m}F(s,u^{m}(s^{-});y)\right\rangle_{H} \right] \nu (\diff y) \diff s .
\end{aligned}
\end{equation}
We have
\begin{equation}\nonumber
\begin{aligned}
\left|u^{ m}(t\wedge\tau_{m}(R))\right|_{H}^{2p} = & \left|P_{n}u_{0}\right|^{2p}_{H} + K^{m}(t) + M^{m}(t) + I^{m}(t) \\
= & \left|P_{n}u_{0}\right|^{2p}_{H} + \int_{0}^{t\wedge\tau_{m}(R)} 2p\left|u^{ m}(s)\right|_{L^{2}(\mathbb{R})}^{2p-2}\left\langle u^{ m}(s) ; \Phi^{m} \big(u^{ m}(s)\big) \diff W^{m}(s) \right\rangle_{H} \\
& - \int_{0}^{t\wedge\tau_{m}(R)} 2p \left|u^{ m}(s)\right|_{L^{2}(\mathbb{R})}^{2p-2} \left\langle u^{ m}(s); \mathscr{K}_{\theta}\big( u^{ m}(s) \big) \right\rangle_{H} \diff s \\
& + \int_{0}^{t\wedge\tau_{m}(R)} p^{2} \left|u^{ m}(s)\right|_{L^{2}(\mathbb{R})}^{2p-2} \left| \Phi^{m} \big(u^{ m}(s)\big) \right|^{2}_{L^{2}_{0}(L^{2}(\mathbb{R}))} \diff s  + M^{m}(t) + I^{m}(t).
\end{aligned}
\end{equation}


Using condition (\ref{P2}) we obtain
\begin{equation}
\begin{aligned}
\left|u^{ m}(t\wedge\tau_{m}(R))\right|_{L^{2}(\mathbb{R})}^{2p} \leq & \left|P_{n}u_{0}\right|^{2p}_{H} + \int_{0}^{t\wedge\tau_{m}(R)}2p\left|u^{ m}(s)\right|_{L^{2}(\mathbb{R})}^{2p-2} \left\langle u^{ m}(s); \Phi(u^{ m}(s)) \diff W^{m}(s) \right\rangle_{H} \\\!\!
& + \int_{0}^{t\wedge\tau_{m}(R)}  p \left|u^{ m}(s)\right|_{L^{2}(\mathbb{R})}^{2p-2} \left(-\alpha  \left|u^{ m}(s)\right|^{2}_{H^{2}(\mathbb{R})} 
+ \beta \left|u^{ m}(s)\right|_{L^{2}(\mathbb{R})}^{2} +\kappa\right)\!\diff s\\ & + M^{m}(t) + I^{m}(t) ,
\end{aligned} \nonumber 
\end{equation}
where $H^{2}(\mathbb{R})$ denotes Sobolev space. Then
\begin{equation}
\begin{aligned}
\left|u^{ m}(t\wedge\tau_{m}(R))\right|_{L^{2}(\mathbb{R})}^{2p} + & \int_{0}^{t\wedge\tau_{m}(R)} \delta p\left|u^{ m}(s)\right|_{L^{2}(\mathbb{R})}^{2p-2}\left|u^{ m}(s)\right|^{2}_{H^{2}(\mathbb{R})} \diff s  \\ 
\leq & \left|P_{n}u_{0}\right|^{2p}_{H} + \int_{0}^{t\wedge\tau_{m}(R)} 2p\left|u^{ m}(s)\right|_{L^{2}(\mathbb{R})}^{2p-2} \left\langle u^{ m}(s); \Phi(u^{ m}(s)) \diff W^{m}(s)\right\rangle_{H} \\
& + \int_{0}^{t\wedge\tau_{m}(R)} p \left|u^{ m}(s)\right|_{L^{2}(\mathbb{R})}^{2p-2} \left((\delta -\alpha) \left|u^{ m}(s)\right|^{2}_{H^{2}(\mathbb{R})} 
+ \beta \left|u^{ m}(s)\right|_{L^{2}(\mathbb{R})}^{2} + \kappa  \right) \diff s\\ & + M^{m}(t) + I^{m}(t).
\end{aligned}
\end{equation}

Moreover, for any $\varepsilon>0$, 
\begin{equation}\nonumber
\begin{aligned}
\left|u^{ m}(t\wedge\tau_{m}(R))\right|_{L^{2}(\mathbb{R})}^{2p} + & \int_{0}^{t\wedge\tau_{m}(R)} \delta(p-p\varepsilon)\left|u^{ m}(s)\right|_{L^{2}(\mathbb{R})}^{2p-2}\left|u^{ m}(s)\right|^{2}_{H^{2}(\mathbb{R})} \diff s  \\ 
\leq & \left|P_{n}u_{0}\right|^{2p}_{H} + \int_{0}^{t\wedge\tau_{m}(R)} 2p\left|u^{ m}(s)\right|_{L^{2}(\mathbb{R})}^{2p-2} \left\langle u^{ m}(s); \Phi(u^{ m}(s)) \diff W^{m}(s)\right\rangle_{H} \\
& + \int_{0}^{t\wedge\tau_{m}(R)} p \left|u^{ m}(s)\right|_{L^{2}(\mathbb{R})}^{2p-2} \left((\delta -\alpha ) \left|u^{ m}(s)\right|^{2}_{H^{2}(\mathbb{R})} 
 +\beta \left|u^{ m}(s)\right|_{L^{2}(\mathbb{R})}^{2} + \kappa\right) \diff s\\ & + M^{m}(t) + I^{m}(t);
\end{aligned}
\end{equation}
\begin{equation}\label{Lpom1}
\begin{aligned}
\left|u^{ m}(s)\right|_{L^{2}(\mathbb{R})}^{2p} + & \int_{0}^{t\wedge\tau_{m}(R)} \left[p (\alpha  -\varepsilon\delta) \right] \left|u^{ m}(s)\right|_{L^{2}(\mathbb{R})}^{2p-2} \left|u^{ m}(s)\right|^{2}_{H^{2}(\mathbb{R})} \diff s \\
\leq & \left|P_{n}u_{0}\right|^{2p}_{H} + \int_{0}^{t\wedge\tau_{m}(R)} 2p\left|u^{ m}(s)\right|_{L^{2}(\mathbb{R})}^{2p-2} \left\langle u^{ m}(s); \Phi(u^{ m}(s)) \diff W^{m}(s)\right\rangle_{H} \\ 
& + \int_{0}^{t\wedge\tau_{m}(R)} p \left|u^{ m}(s)\right|_{L^{2}(\mathbb{R})}^{2p-2} \left(\beta  \left|u^{ m}(s)\right|_{L^{2}(\mathbb{R})}^{2} + \kappa \right) \diff s + M^{m}(t) + I^{m}(t).
\end{aligned}
\end{equation}

Substitution in the Young inequality (e.g., see inequality 8.3 in \cite{BDG})  $a:=\left|u^{ m}(t,x)\right|_{L^{2}(\mathbb{R})}^{2p-2}$, $b:=\kappa  p$, $r:=\frac{p}{p-1}$, $r'=p$, gives
\begin{equation}\nonumber
\begin{aligned}\frac{\varepsilon}{r} = & \frac{\varepsilon}{\frac{2p}{2p-2}} = \frac{\varepsilon (p-1)}{p}, \\
\frac{r'}{r} = & \frac{p}{\frac{p}{p-1}} = p-1, \\
ab = & \kappa  p \left|u^{ m}(t,x)\right|_{L^{2}(\mathbb{R})}^{2p-2} 
\end{aligned}
\end{equation} 
and 
\begin{equation}\label{LYN}
\begin{aligned}
\kappa  p \left|u^{ m}(t,x)\right|_{L^{2}(\mathbb{R})}^{2p-2}  \leq & \frac{\varepsilon}{\frac{p}{p-1}}\left|u^{ m}(t,x)\right|_{L^{2}(\mathbb{R})}^{2p} + \frac{1}{p\varepsilon^{p-1}}(\kappa  p)^{p} \\
= & C_{1}(\varepsilon, p) \left|u^{ m}(t,x)\right|_{L^{2}(\mathbb{R})}^{2p} + C_{2}(\varepsilon, \kappa , p).
\end{aligned}
\end{equation}

Using (\ref{LYN}) in (\ref{Lpom1}) one gets
\begin{equation}\label{L*}
\begin{aligned}
\left|u^{ m}(t\wedge\tau_{m}(R))\right|_{L^{2}(\mathbb{R})}^{2p} + & \int_{0}^{t\wedge\tau_{m}(R)}\left[p (\alpha  -\varepsilon\delta) \right] \left|u^{ m}(s)\right|_{L^{2}(\mathbb{R})}^{2p-2} \left|u^{ m}(s)\right|^{2}_{H^{2}(\mathbb{R})} \diff s \\
\leq & \left|P_{n}u_{0}\right|^{2p}_{H} + \int_{0}^{t\wedge\tau_{m}(R)} 2p\left|u^{ m}(s)\right|_{L^{2}(\mathbb{R})}^{2p-2} \left\langle u^{ m}(s); \Phi(u^{ m}(s)) \diff W^{m}(s)\right\rangle_{H} \\ 
& + \int_{0}^{t\wedge\tau_{m}(R)} \left(p\beta  + C_{1}(\varepsilon, p)\right)\left|u^{ m}(s)\right|_{L^{2}(\mathbb{R})}^{2p} \diff s + \int_{0}^{t\wedge\tau_{m}(R)} \!\!C_{2}(\varepsilon, \kappa , p) \diff s \\
&+ M^{m}(t) + I^{m}(t) \\
=& \int_{0}^{t\wedge\tau_{m}(R)} 2p\left|u^{ m}(s)\right|_{L^{2}(\mathbb{R})}^{2p-2} \left\langle u^{ m}(s); \Phi(u^{ m}(s)) \diff W^{m}(s)\right\rangle_{H} \\
&+ \int_{0}^{t\wedge\tau_{m}(R)} C_{3}(\varepsilon, \beta , p) \left|u^{ m}(s)\right|_{L^{2}(\mathbb{R})}^{2p} \diff s + \int_{0}^{t\wedge\tau_{m}(R)} C_{2}(\varepsilon, \kappa , p) \diff s\\
&+ M^{m}(t) + I^{m}(t).
\end{aligned}
\end{equation}

Let $\varepsilon < \frac{\alpha}{\delta}$ be arbitrary fixed. Then
\begin{equation}\nonumber
\begin{aligned}
& \left|u^{ m}(t\wedge\tau_{m}(R))\right|_{L^{2}(\mathbb{R})}^{2p} \leq  \left|P_{m}u_{0}(x)\right|_{L^{2}(\mathbb{R})}^{2p} + \int_{0}^{t\wedge\tau_{m}(R)} p \left|u^{ m}(s)\right|_{L^{2}(\mathbb{R})}^{2p-2} \left\langle u^{ m}(s); \Phi(u^{ m}(s)) \diff W^{m}(s)\right\rangle_{H} \\
& + \int_{0}^{t\wedge\tau_{m}(R)} C_{3}(\varepsilon, \beta , p) \left|u^{ m}(s)\right|_{L^{2}(\mathbb{R})}^{2p} \diff s + \int_{0}^{t} C_{2}(\varepsilon, \kappa , p) \diff s + M^{m}(t) + I^{m}(t)
\end{aligned}
\end{equation}
and
\begin{equation}\label{Lszac1}
\begin{aligned}
\mathbb{E} & \left|u^{ m}(t\wedge\tau_{m}(R))\right|_{L^{2}(\mathbb{R})}^{2p} \leq  C_{4} + \mathbb{E} \int_{0}^{t\wedge\tau_{m}(R)} C_{3}(\varepsilon, \beta , p) \left|u^{ m}(s)\right|_{L^{2}(\mathbb{R})}^{2p} \diff s  \\ 
& + \mathbb{E} \int_{0}^{t\wedge\tau_{m}(R)} C_{2}(\varepsilon, \kappa , p) \diff s + \mathbb{E} (M^{m}(t)) + \mathbb{E}(I^{m}(t)). 
\end{aligned}
\end{equation}

In the following part we will use the following result from \cite{Mot}.
\begin{lemma}(\cite[p.~882-883]{Mot}) \label{Taylor} 
For any $p\geq 1$ there exist constants $C_{1}(p), C_{2}(p), C_{3}(p)>0$, such that for aarbitrary  $x,h\in H$ the following inequalities hold
\begin{equation}\nonumber
\begin{aligned}
\Big| \left| x+h \right|^{2p}_{H} - \left|x\right|^{2p}_{H} - 2p \left|x\right|^{2p-2}_{H} \left\langle x,h\right\rangle_{H} \Big| \leq & C_{1}(p) \left(\left|x\right|^{2p-2}_{H} + \left|h\right|^{2p-2}_{H}\right) \left|h\right|^{2p}_{H} ; 
\end{aligned}
\end{equation}
\begin{equation}\nonumber
\begin{aligned}
\left(\left|x+h\right|^{2p}_{H} - \left|x\right|^{2p}_{H}\right)^{2} \leq & 2 \left\{4p^{2}\left|x\right|^{2p-2}_{H}\left|h\right|^{2}_{H} + C_{2}(p)\left(\left|x\right|^{2p-2}_{H} + \left|h\right|^{2p-2}_{H}\right)^{2}\left|h\right|^{4}_{H}\right\} \\
\leq & 4p^{2} \left|x\right|^{4p-2}_{H}\left|h\right|^{2}_{H} + C_{3}(p)\left|x\right|^{4p-4}_{H}\left|h\right|^{4}_{H}+C_{3}(p)\left|h\right|^{4p}_{H}.
\end{aligned}
\end{equation}
\end{lemma}
Using condition (\ref{F3}), Lemma \ref{Taylor} and (\ref{stop}) for the process $I^{n}(t\wedge\tau_{m}(R))$ one obtains 
\begin{equation}\nonumber
\begin{aligned}
\left|I^{n}(t) \right|:= & \Bigg| \int_{0}^{t\wedge\tau_{m}(R)} \int_{Y} \left[ \left|u^{m}(s^{-})+P_{m}F(s,u^{m}(s^{-});y)\right|^{2p}_{H} - \left|u^{m}(s^{-})\right|^{2p}_{H} \right. \\
& \left. - 2p \left|u^{m}(s^{-})\right|^{2p-2}_{H} \left\langle u^{m}(s^{-}); P_{m}F(s,u^{m}(s^{-});y)\right\rangle_{H}\right] \nu (\diff y) \diff s \Bigg| \\
\leq &  \int_{0}^{t\wedge\tau_{m}(R)} \int_{Y} \left\{C_{4}(p) \left( \left|u^{m}(s)+P_{m}F(s,u^{m}(s);y)\right|^{2p-2}_{H} + \left|P_{m}F(s,u^{m}(s);y)\right|^{2p-2}_{H} \right) \right. \\
& \times \left.\left|P_{m}F(s,u^{m}(s);y)\right|^{2p}_{H} \right\} \nu (\diff y) \diff s \\
\leq & C_{4}(p) \int_{0}^{t} \left\{C_{5} \left|u^{m}(s)\right|^{2p-2}_{H} \left(1+\left|u^{m}(s)\right|^{2}_{H}\right) + C_{6}(p)\left(1+\left|u^{m}(s)\right|^{2p}_{H}\right)\right\} \diff s \\
\leq & C_{7}(p) \int_{0}^{t} \left\{ 1 + \left|u^{m}(s)\right|^{2p}_{H} \right\} \diff s = C_{7}(p)t + C_{7}(p)\int_{0}^{t}\left|u^{m}(s)\right|^{2p}_{H} \diff s
\end{aligned}
\end{equation}
for some $C_{4}(p), C_{5}, C_{6}(p), C_{7}(p) >0$. Then for any $t\in[0,T]$
\begin{equation}\label{EIszac}
\mathbb{E}\left(\left|I^{m}(t)\right|\right) \leq C_{7}(p)t + C_{7}(p) \int_{0}^{t} \mathbb{E} \left( \left|u^{m}(s)\right|^{2p}_{H} \right) \diff s.
\end{equation}

Moreover, due to Lemma
 \ref{Taylor}, condition (\ref{F3}) and (\ref{stop}) the process 
$\left\{M^{m}(t\wedge\tau_{m}(R))\right\}_{t\in[0,T]}$ is an integrable martingale, so
\begin{equation}\label{EMszac}
\mathbb{E}\left( M^{m}(t\wedge\tau_{m}(R)) \right) = 0.
\end{equation}
Insertion of  (\ref{EIszac}) and (\ref{EMszac}) into (\ref{Lszac1}) yields
\begin{equation}\label{lszac2}
\begin{aligned}
\mathbb{E} \left|u^{ m}(t\wedge\tau_{m}(R))\right|_{H}^{2p} \leq & C_{4} + \mathbb{E} \int_{0}^{t\wedge\tau_{m}(R)} C_{3}(\varepsilon, \beta , p) \left|u^{ m}(s)\right|_{H}^{2p} \diff s + \mathbb{E} \int_{0}^{t\wedge\tau_{m}(R)} C_{2}(\varepsilon, \kappa , p) \diff s \\ 
&  + C_{7}(p)t + C_{7}(p) \int_{0}^{t} \mathbb{E} \left( \left|u^{m}(s)\right|^{2p}_{H} \right) \diff s \\
\leq & C_{4} + \mathbb{E} \int_{0}^{T\wedge\tau_{m}(R)} C_{3}(\varepsilon, \beta , p) \left|u^{ m}(s)\right|_{H}^{2p} \diff s + \mathbb{E} \int_{0}^{t\wedge\tau_{m}(R)} C_{2}(\varepsilon, \kappa , p) \diff s \\ 
&  + C_{7}(p)t + C_{7}(p) \int_{0}^{T} \mathbb{E} \left( \left|u^{m}(s)\right|^{2p}_{H} \right) \diff s \\
\leq & C_{4} + C_{8}T + \int_{0}^{T\wedge\tau_{m}(R)}C_{9}(\varepsilon, \beta, p) \mathbb{E} \left( \left|u^{m}(s)\right|^{2p}_{H} \right) \diff s .
\end{aligned}
\end{equation}


Substitute in the Gronwall lemma (e.g., see Theorem 1.2 in \cite{Hart})  $u(t):= \mathbb{E} \left|u^{ m}(t\wedge\tau_{m}(R))\right|_{H}^{2p}$, $\alpha(t):=C_{4} + C_{8}T$, $\beta(t):\equiv C_{9}(\varepsilon, \beta, p)$, $a:=0$, $b:=T\wedge\tau_{m}(R)$. Then for any $t\in[0,T\wedge\tau_{m}(R)]$ and $m \in \mathbb{N}\setminus \left\{0\right\}$
\begin{equation}\label{LGN}
\begin{aligned}
&\mathbb{E}\left|u^{ m}(t\wedge\tau_{m}(R))\right|_{H}^{2p} \\
\leq & C_{4} + C_{8}T + \left|\int_{0}^{t\wedge\tau_{m}(R)} \left[C_{4} + C_{8}s\right]C_{9}(\varepsilon, \beta, p) \exp\left\{\left|\int_{s}^{t\wedge\tau_{m}(R)} C_{9}(\varepsilon, \beta, p) \diff \xi\right|\right\} \diff s\right| \\
\leq &  C_{10}(\varepsilon, \beta, p, T)
\end{aligned}
\end{equation}
for some constant $C_{10}(\varepsilon, \beta, p, T) > 0$. Moreover, 
\begin{equation}\nonumber
\sup_{n\geq 1}\sup_{t\in[0,T]} \mathbb{E}\left|u^{ m}(t\wedge\tau_{m}(R))\right|_{H}^{2p} \leq C_{10}(\varepsilon, \beta, p, T)
\end{equation}
and in particular
\begin{equation}\nonumber
\sup_{n\geq 1} \mathbb{E}\left[\int_{0}^{T\wedge\tau_{m}(R)}\left|u^{ m}(s)\right|_{H}^{2p} \diff s\right] \leq C_{10}(\varepsilon, \beta, p, T) . 
\end{equation}
Due to this inequality, when $R\uparrow\infty$ the following inequality holds
\begin{equation}\label{supszac}
\sup_{n\geq 1} \mathbb{E}\left[\int_{0}^{T}\left|u^{ m}(s)\right|_{H}^{2p} \diff s\right] \leq C_{10}(\varepsilon, \beta, p, T) . 
\end{equation} 
Using (\ref{EIszac}), (\ref{EMszac}) and (\ref{supszac}) in (\ref{L*}) one gets 
\begin{equation}\nonumber
\begin{aligned}
\sup_{m\geq 1}  & \mathbb{E} \left[\left|u^{ m}(t\wedge\tau_{m}(R))\right|_{H}^{2p} \right] +  \sup_{m\geq 1}  \mathbb{E} \left[ \int_{0}^{t\wedge\tau_{m}(R)}\left[p (\alpha  -\varepsilon\delta) \right] \left|u^{ m}(s)\right|_{H}^{2p-2} \left|u^{ m}(s)\right|^{2}_{V} \diff s \right] \\
\leq & \sup_{m\geq 1}  \mathbb{E} \left[ \int_{0}^{t\wedge\tau_{m}(R)} 2p\left|u^{ m}(s)\right|_{H}^{2p-2} \left\langle u^{ m}(s); \Phi(u^{ m}(s)) \diff W^{m}(s)\right\rangle_{H} \right] \\
&+ \sup_{m\geq 1}  \mathbb{E} \left[\int_{0}^{t\wedge\tau_{m}(R)} C_{3}(\varepsilon, \beta , p) \left|u^{ m}(s)\right|_{H}^{2p} \diff s +\int_{0}^{t\wedge\tau_{m}(R)}  C_{2}(\varepsilon, \kappa , p) \diff s \right]\\
&+ \sup_{m\geq 1}  \mathbb{E} \left[M^{m}(t) \right] + \sup_{m\geq 1}  \mathbb{E} \left[ I^{m}(t) \right] \\
\leq & \sup_{m\geq 1}  \mathbb{E} \left[\int_{0}^{T} C_{3}(\varepsilon, \beta , p) \left|u^{ m}(s)\right|_{H}^{2p} \diff s + \int_{0}^{T} C_{2}(\varepsilon, \kappa , p) \diff s \right]\\
&+ \sup_{m\geq 1}  \mathbb{E} \left[M^{m}(t) \right] + \sup_{m\geq 1}  \mathbb{E} \left[ I^{m}(t) \right] \\
\leq & C_{10}(\varepsilon, \beta, p, T)C_{3}(\varepsilon, \beta , p) + C_{2}(\varepsilon, \kappa , p)T + C_{7}(p)T + C_{7}(p) \sup_{m\geq 1} \int_{0}^{T} \mathbb{E} \left( \left|u^{m}(s)\right|^{2p}_{H} \right) \diff s \\
\leq & C_{11}(\varepsilon, \beta, \kappa, p, T) + C_{7}(p) C_{10}(\varepsilon, \beta, p, T) \leq C_{12}(\varepsilon, \beta, \kappa, p, T) .
\end{aligned}
\end{equation}
Substitution in the above inequality  $p:=1$, give for any  $t\in[0,T]$
\begin{equation}\nonumber
\begin{aligned}
\sup_{m\geq 1}  & \mathbb{E} \left[\left|u^{ m}(t\wedge\tau_{m}(R))\right|_{H}^{2} \right] +  \sup_{m\geq 1}  \mathbb{E} \left[ \int_{0}^{t\wedge\tau_{m}(R)}C_{13}(\varepsilon, \alpha, \delta)  \left|u^{ m}(s)\right|^{2}_{V} \diff s \right] \leq C_{12}(\varepsilon, \beta, \kappa, 1, T), \\
\sup_{m\geq 1} &  \mathbb{E} \left[ \int_{0}^{T}C_{13}(\varepsilon, \alpha, \delta)  \left|u^{ m}(s)\right|^{2}_{V} \diff s \right] + C_{10}(\varepsilon, \beta, 1, T) \leq C_{12}(\varepsilon, \beta, \kappa, , T), \\
\sup_{m\geq 1} &  \mathbb{E} \left[ \int_{0}^{T}C_{13}(\varepsilon, \alpha, \delta)  \left|u^{ m}(s)\right|^{2}_{V} \diff s \right] \leq - C_{10}(\varepsilon, \beta, 1, T) + C_{12}(\varepsilon, \beta, \kappa, 1, T) \leq C_{14}(\varepsilon, \beta, \kappa, 1, T), \\
\sup_{m\geq 1} &  \mathbb{E} \left[ \int_{0}^{T} \left|u^{ m}(s)\right|^{2}_{V} \diff s   \leq C_{14}(\varepsilon, \beta, \kappa, T)\right].
\end{aligned}
\end{equation}
Since  $\varepsilon, \beta, \kappa, T$ are fixed, it gives (\ref{LevyC2}).

From the Burkholder lemma (e.g., see Theorem 2.3 in \cite{BDG}) for the process $M^{m}(t)$ one obtains 
\begin{equation}\label{BurkM}
\begin{aligned}
\mathbb{E} & \left[\sup_{r\in[0,t]}\left|M^{n}(r\wedge\tau_{m}(R))\right|\right] \\
\leq & C_{15}(p)\mathbb{E}\left[\left(\int_{0}^{t\wedge\tau_{m}(R)}\int_{Y}\left\{\left|u^{m}(s^{-}) + P_{m}F(s,u^{m}(s^{-});y)\right|^{2p}_{H} - \left|u^{m}(s^{-})\right|^{2p}_{H}\right\}\nu(\diff y)\diff s\right)^{\frac{1}{2}}\right]
\end{aligned}
\end{equation}
for some $C_{15}(p)>0$. Moreover, due to condition (\ref{F3}), Lemma \ref{Taylor} for some  $C_{16},C_{17},C_{18},C_{19}>0$ there holds
\begin{equation}\label{BurkFpom}
\begin{aligned}
\int_{Y} & \left(u^{m}(s^{-}) + P_{m}F(s,u^{m}(s^{-});y)\right|^{2p}_{H} - \left|u^{m}(s^{-})\right|^{2p}_{H}\nu(\diff y) \\
\leq & C_{16} + C_{17} \left|u^{m}(s^{-})\right|_{H}^{4p-4} + C_{18}\left|u^{m}(s^{-})\right|^{4p-2}_{H} + C_{19}\left|u^{m}(s^{-})\right|^{4p}_{H} .
\end{aligned}
\end{equation}
 Young's inequality in (\ref{BurkFpom}) implies
\begin{equation}\label{BurkFpom1}
\begin{aligned}
\int_{Y} & \left(u^{m}(s^{-}) + P_{m}F(s,u^{m}(s^{-});y)\right|^{2p}_{H} - \left|u^{m}(s^{-})\right|^{2p}_{H}\nu(\diff y) \\
\leq & C_{20} + C_{21} \left|u^{m}(s^{-})\right|_{H}^{4p} .
\end{aligned}
\end{equation}
for some constants $C_{20},C_{21}>0$. Therefore
\begin{equation}\label{BurkFpom2}
\begin{aligned}
& \left(\int_{0}^{t\wedge\tau_{m}(R)}\int_{Y}\left\{\left|u^{m}(s^{-}) + P_{m}F(s,u^{m}(s^{-});y)\right|^{2p}_{H} - \left|u^{m}(s^{-})\right|^{2p}_{H}\right\}\nu(\diff y)\diff s\right)^{\frac{1}{2}} \\
\leq & \sqrt{TC_{20}} + \sqrt{C_{21}} \left( \int_{0}^{t\wedge\tau_{m}(R)} \left|u^{m}(s^{-})\right|_{H}^{4p} \diff s\right)^{\frac{1}{2}}.
\end{aligned}
\end{equation}
Using (\ref{supszac}) and (\ref{BurkFpom2}) in (\ref{BurkM}), one get for some constants  $C_{22},C_{23}>0$
\begin{equation}\label{BurkFpom3}
\begin{aligned}
\mathbb{E} & \left[\sup_{r\in[0,t]}\left|M^{n}(r\wedge\tau_{m}(R))\right|\right] \leq C_{15}(p) \sqrt{TC_{20}} + C_{15}(p) \sqrt{C_{21}} \mathbb{E}\left[ \left( \int_{0}^{t\wedge\tau_{m}(R)} \left|u^{m}(s^{-})\right|_{H}^{4p} \diff s\right)^{\frac{1}{2}} \right] \\
& \leq  C_{15}(p) \sqrt{TC_{20}} + C_{15}(p) \sqrt{C_{21}} \mathbb{E}\left[\left(\sup_{s\in[0,t]} \left|u^{m}(s^{-})\right|_{H}^{2p}\right) \left( \int_{0}^{t\wedge\tau_{m}(R)} \left|u^{m}(s^{-})\right|_{H}^{2p} \diff s\right)^{\frac{1}{2}} \right] \\
& \leq  C_{15}(p) \sqrt{TC_{20}} + \frac{1}{4} \mathbb{E}\left[\left(\sup_{s\in[0,t]} \left|u^{m}(s^{-})\right|_{H}^{2p}\right)\right] + C_{22}(p) \sqrt{C_{21}} \mathbb{E}\left[ \left( \int_{0}^{t\wedge\tau_{m}(R)} \left|u^{m}(s^{-})\right|_{H}^{2p} \diff s\right)^{\frac{1}{2}} \right] \\
& \leq \frac{1}{4} \mathbb{E}\left[\left(\sup_{s\in[0,t]} \left|u^{m}(s^{-})\right|_{H}^{2p}\right)\right] + C_{23}(p) .
\end{aligned}
\end{equation}

Moreover, using Burkholder's inequality for the process $$\int_{0}^{t\wedge\tau_{m}(R)}p\left|u^{ m}(s)\right|_{L^{2}(\mathbb{R})}^{2p-2}\left\langle \Phi(u^{ m}(s)) \diff W^{m}(s), u^{ m}(s)\right\rangle_{H}$$ one obtains for some constant $C_{23}(p)>0$ 
\begin{equation}\label{PhiBurk}
\begin{aligned}
\mathbb{E} & \left( \sup_{t\in[0,T]}\int_{0}^{t\wedge\tau_{m}(R)}p\left|u^{ m}(s)\right|_{H}^{2p-2}\left\langle \Phi(u^{ m}(s)) \diff W^{m}(s), u^{ m}(s)\right\rangle_{H} \right) \\
\leq & C_{23}(p) \mathbb{E} \left\{ \left[ \int_{0}^{t\wedge\tau_{m}(R)}  p\left|u^{ m}(s)\right|_{H}^{4p-2} \left\|\Phi(u^{ m}(s))\right\|^{2}_{L_{0}^{2}(H)}  \diff s \right]^{\frac{1}{2}} \right\} \\
\end{aligned}
\end{equation}
\begin{equation} \nonumber 
\begin{aligned}	
\leq & C_{23}(p) p \mathbb{E} \left\{ \bigg[ \sup_{0 \leq s \leq T} \left|u^{ m}(s)\right|_{H}^{2p}  \int_{0}^{t\wedge\tau_{m}(R)} \left|u^{ m}(s)\right|_{H}^{2p-2} \left\|\Phi(u^{ m}(s))\right\|^{2}_{L_{0}^{2}(H)} \diff s \bigg]^{\frac{1}{2}} \right\} \\
\leq & C_{23}(p) p \mathbb{E} \left\{ \bigg[ \sup_{0 \leq s \leq T} \left|u^{ m}(s)\right|_{H}^{2p}  \int_{0}^{t\wedge\tau_{m}(R)} \left|u^{ m}(s)\right|_{H}^{2p-2} \left( C_{\Phi}\left|u^{ m}(s) \right|^{2}_{H} + 1 \right)  \diff s \bigg]^{\frac{1}{2}} \right\} \\
\leq & \frac{1}{2}\mathbb{E}\left( \sup_{0\leq s\leq T} \left| u^{ m}(s) \right|_{H}^{2p} \right) + \frac{1}{2} C_{23}(p)^{2} p^{2} \mathbb{E} \left( \int_{0}^{t\wedge\tau_{m}(R)} C_{\Phi} \sup_{0\leq s\leq \xi} \left|u^{ m}(s)\right|_{H}^{2p}\diff \xi \right) \\
& + \frac{1}{2}C_{23}(p)^{2}p^{2}1 \mathbb{E} \int_{0}^{t\wedge\tau_{m}(R)} \left|u^{ m}(s)\right|^{2p-2} \diff s \\
\leq & \frac{1}{2}\mathbb{E}\left( \sup_{0\leq s\leq T} \left| u^{ m}(s) \right|_{H}^{2p} \right) + \frac{1}{2} C_{23}(p)^{2} p^{2} \mathbb{E} \left( \int_{0}^{t\wedge\tau_{m}(R)} C_{\Phi} \sup_{0\leq s\leq \xi} \left|u^{ m}(s)\right|_{H}^{2p}\diff \xi \right) + \frac{1}{2}C_{23}(p)^{2}p^{2} \bar{C}_{p}T \\
\leq & \frac{1}{2}\mathbb{E} \left( \sup_{0\leq s\leq T} \left|u^{ m}(s)\right|_{H}^{2p} \right) + C_{15}( p)\mathbb{E}\left(\int_{0}^{t\wedge\tau_{m}(R)} \sup_{0\leq s\leq T} \left| u^{ m}(s) \right|_{H}^{2p} \diff \xi\right) + C_{16}( p, T) \\
\leq & \frac{1}{2}\mathbb{E} \left( \sup_{0\leq s\leq T} \left|u^{ m}(s)\right|_{H}^{2p} \right) + C_{24}(p,T) \mathbb{E}\left[ \sup_{0\leq s\leq T} \left| u^{ m}(s) \right|_{H}^{2p} \right]+ C_{16}(p,T) \\
\leq & C_{25}(p,T)\mathbb{E}\left[\sup_{0\leq s\leq T} \left| u^{ m}(s) \right|_{H}^{2p}\right] + C_{16}(p,T) .
\end{aligned}
\end{equation} 
Now, we have
\begin{equation}\nonumber
\begin{aligned}
\left|u^{ m}(t\wedge\tau_{m}(R))\right|_{L^{2}(\mathbb{R})}^{2p}\leq& \int_{0}^{t\wedge\tau_{m}(R)} 2p\left|u^{ m}(s)\right|_{L^{2}(\mathbb{R})}^{2p-2} \left\langle u^{ m}(s), \Phi(u^{ m}(s)) \diff W^{m}(s)\right\rangle_{H} \\
&+ \int_{0}^{t\wedge\tau_{m}(R)} C_{3}(\varepsilon, \beta , p) \left|u^{ m}(s)\right|_{L^{2}(\mathbb{R})}^{2p} \diff s + \int_{0}^{t\wedge\tau_{m}(R)} C_{2}(\varepsilon, \kappa , p) \diff s\\
&+ M^{m}(t) + I^{m}(t).
\end{aligned}
\end{equation}
Taking supremum from the r.h.s. of the above inequality, taking expectation values and using  (\ref{BurkFpom3}) and (\ref{PhiBurk}) one obtains
\begin{equation}\nonumber
\begin{aligned}
\mathbb{E} & \sup_{t\in[0,T]}  \left[\left|u^{ m}(t\wedge\tau_{m}(R))\right|_{L^{2}(\mathbb{R})}^{2p}\right]\leq \mathbb{E}\sup_{t\in[0,T]} \left[ \int_{0}^{t\wedge\tau_{m}(R)} 2p\left|u^{ m}(s)\right|_{L^{2}(\mathbb{R})}^{2p-2} \left\langle u^{ m}(s), \Phi(u^{ m}(s)) \diff W^{m}(s)\right\rangle_{H} \right] \\
&+ \mathbb{E}\sup_{t\in[0,T]} \left[ \int_{0}^{t\wedge\tau_{m}(R)} C_{3}(\varepsilon, \beta , p) \left|u^{ m}(s)\right|_{L^{2}(\mathbb{R})}^{2p} \diff s \right] + \mathbb{E}\sup_{t\in[0,T]} \left[ \int_{0}^{t\wedge\tau_{m}(R)} C_{2}(\varepsilon, \kappa , p) \diff s \right]\\
&+ \mathbb{E}\sup_{t\in[0,T]} \left[M^{m}(t)\right] + \mathbb{E}\sup_{t\in[0,T]} \left[I^{m}(t)\right] \\
\leq &+ C_{26} + C_{25}(p,T)\mathbb{E}\left[\sup_{0\leq s\leq T} \left| u^{ m}(s) \right|_{H}^{2p}\right] + C_{16}(p,T) + \mathbb{E}\sup_{t\in[0,T]} \left[ \int_{0}^{t\wedge\tau_{m}(R)} C_{3}(\varepsilon, \beta , p) \left|u^{ m}(s)\right|_{L^{2}(\mathbb{R})}^{2p} \diff s \right] \\
& + \frac{1}{4} \mathbb{E}\left[\left(\sup_{s\in[0,t]} \left|u^{m}(s^{-})\right|_{H}^{2p}\right)\right] + C_{23}(p) + \mathbb{E}\sup_{t\in[0,T]} \left[I^{m}(t)\right] \\
\leq & C_{27}(p,T) + C_{28}(p,T)\mathbb{E}\left[\sup_{0\leq s\leq T} \left| u^{ m}(s) \right|_{H}^{2p}\right] +C_{3}(\varepsilon, \beta , p) \mathbb{E}\sup_{t\in[0,T]} \left[ \int_{0}^{t\wedge\tau_{m}(R)} \left|u^{ m}(s)\right|_{L^{2}(\mathbb{R})}^{2p} \diff s \right] \\
& + \mathbb{E}\sup_{t\in[0,T]} \left[I^{m}(t)\right]. \\
\end{aligned}
\end{equation}
Inequalities (\ref{EIszac}), (\ref{supszac}) imply
\begin{equation}
\begin{aligned}
\mathbb{E}\sup_{t\in[0,T]}  & \left[\left|u^{ m}(t\wedge\tau_{m}(R))\right|_{L^{2}(\mathbb{R})}^{2p}\right] \leq C_{29}(p,T) + C_{30}(p,T) + C_{31}(p) \int_{0}^{t} \mathbb{E} \left( \left|u^{m}(s)\right|^{2p}_{H} \right) \diff s + C_{32}(\varepsilon, \beta, p, T) \\
\leq & C_{29}(p,T) + C_{30}(p,T) + C_{31}(p) C_{32}(\varepsilon, \beta, p, T) + C_{32}(\varepsilon, \beta, p, T) \leq C_{33}(\varepsilon, \beta, p, T) .
\end{aligned}
\end{equation}
Taking the limit  $R\uparrow\infty$ yields (\ref{LevyC1}).

Now, let $p\in\left[\right.\frac{1}{2}, 2+\frac{\varsigma}{2}\left.\right)\setminus\left\{2\right\}$ and let $m\in\mathbb{N}$ be arbitrary fixed. Then
\begin{equation}\nonumber
\begin{aligned}
\left|u^{m}(s)\right|^{2p}_{H} = \left( \left|u^{m}(s)\right|^{2+\frac{\varsigma}{2}}_{H}\right)^{\frac{2p}{2+\frac{\varsigma}{2}}} \leq \left(\sup_{t\in[0,T]} \left|u^{m}(s)\right|^{2+\frac{\varsigma}{2}}_{H}\right)^{\frac{2p}{2+\frac{\varsigma}{2}}}
\end{aligned}
\end{equation}
and 
\begin{equation}\nonumber
\begin{aligned}
\mathbb{E}\left[\sup_{t\in[0,T]}\left|u^{m}(t)\right|^{2p}_{H}\right] \leq & \mathbb{E}\left[\left(\sup_{t\in[0,T]} \left|u^{m}(s)\right|^{2+\frac{\varsigma}{2}}_{H}\right)^{\frac{2p}{2+\frac{\varsigma}{2}}} \right] \leq \left( \mathbb{E}\left[\sup_{t\in[0,T]} \left|u^{m}(s)\right|^{2+\frac{\varsigma}{2}}_{H} \right] \right)^{\frac{2p}{2+\frac{\varsigma}{2}}}\\
\leq & \left[C_{34}\left(4+\frac{\varsigma}{2}\right)\right]^{\frac{2p}{2+\frac{\varsigma}{2}}}.
\end{aligned}
\end{equation}
Since $m\in\mathbb{N}$ is fixed, so
\begin{equation}\nonumber
\mathbb{E}\left[\sup_{t\in[0,T]}\left|u^{m}(t)\right|^{2p}_{H}\right] \leq C_{35}(p),
\end{equation}
what finishes the proof.
\end{proof}

\begin{proof}[Proof of  Lemma \ref{L5.4}] 
For reader's convenience we cite lemmas from \cite{Mot} explicitly.
\begin{lemma}(\cite[Corollary 3.5, tightness criterium]{Mot} )\label{LevyCiasKr}
Let $\left\{X_{n}\right\}_{n\in\mathbb{N}}$ be a sequence of processes of c\`{a}dl\`{a}g type, adapted to filtration $\left\{\mathscr{F}\right\}_{t\geq 0}$ with values in  $U'$, such that
\begin{itemize}
\item[(i)] There exists a constant $C_{1}>0$, such that $\sup_{n\in\mathbb{N}}\mathbb{E}\left[\sup_{s\in[0,T]}\left|X_{n}(s)\right|_{H}\right] \leq C_{1}$;
\item[(ii)] There exists a constant  $C_{2}>0$, such that $\sup_{n\in\mathbb{N}}\mathbb{E}\left[\int_{0}^{T}\left|X_{n}\right|^{2}_{V} \diff s\right] \leq C_{2}$;
\item[(iii)] $\left\{X_{n}\right\}$ fulfils the Aldous condition in  $U'$.
\end{itemize}
Then the family of distributions  $\left\{\mathscr{L}(X_{n})\right\}$ is tight in $\mathscr{Z}$.
\end{lemma}

\begin{lemma}(\cite[ Lemma ~6.3]{Mot})\label{Aldous} Let $\left(E,\left|\cdot\right|_{E}\right)$ be a separable Banach space and let  
 $\left\{X_{n}\right\}_{n\in\mathbb{N}}$ be a sequence of random variables with values in  $E$. Let for any sequence of stopping times 
 $\left\{\tau_{n}\right\}_{n\in\mathbb{N}}$, $\tau_{n}<T$, $n\in\mathbb{N}$ and all  $n\in\mathbb{N}$ and $\vartheta > 0$
\begin{equation}\nonumber
\mathbb{E}\left[\left|X_{n}\left(\tau_{n}+\theta\right) - X_{n}(\tau_{n})\right|_{E}^{a}\right] \leq C\vartheta^{b}
\end{equation}
 holds for some $a,b,C>0$. Then the sequence $\left\{X_{n}\right\}_{n\in\mathbb{N}}$ fulfils the Aldous condition in $E$.
\end{lemma}
Let us note, that due to Lemma
 \ref{L5.3} the process $u^{m}(t)$ fulfils conditions  (i) and (ii) from Lemma  \ref{LevyCiasKr} for any $m\in\mathbb{N}$. Then it is sufficient to show that for any  $m\in\mathbb{N}$, $u^{m}(t)$ fulfils Aldous condition. 
We have
\begin{equation}\nonumber
\begin{aligned}
u^{m}(t) = & P_{m}u_{0}(t) - \int_{0}^{t} u^{m}_{3x}(s) \diff s - \int_{0}^{t} u^{m}(s)u^{m}_{x}(s) \diff s + \int_{0}^{t}\int_{Y}P_{m}F(s,u^{m}(s^{-});y)\tilde{\eta}(\diff s, \diff y) \\
& + \int_{0}^{t} P_{n}\Phi(s,u^{m}(s))\diff W(s) .
\end{aligned}
\end{equation}
We will show that each of terms in the above equation fulfils assumptions of the Lemma \ref{Aldous}. 
Let $\vartheta > 0$ and let $\left\{\tau_{m}\right\}_{m\in\mathbb{N}}$ be a sequence of stopping times such that  $\tau_{m}<T$, $m\in\mathbb{N}$. Since  $V \subset H \subset V_{3}'\subset V' \subset U'$, so
\begin{equation}
\begin{aligned}
\mathbb{E} & \left[ \left|u^{m}_{3x}(\tau_{m}+\vartheta) - u^{m}_{3x}(\tau_{m})\right|_{U'} \right] =  \mathbb{E} \left[\left|\int_{\tau_{m}}^{\tau_{m}+\vartheta} u^{m}_{3x}(s) \diff s \right|_{U'} \right] \leq \mathbb{E} \left[\left|\int_{\tau_{m}}^{\tau_{m}+\vartheta} u^{m}_{3x}(s) \diff s \right|_{V_{3}'} \right] \\
\leq & C_{1} \mathbb{E} \left[\left|\int_{\tau_{m}}^{\tau_{m}+\vartheta} u^{m}(s) \diff s \right|_{H} \right] \leq C_{1}C_{2} \mathbb{E} \left[\left|\int_{\tau_{m}}^{\tau_{m}+\vartheta} u^{m}(s) \diff s \right|_{V} \right] \leq C_{1}C_{2}\mathbb{E} \left[\int_{\tau_{m}}^{\tau_{m}+\vartheta} \left|u^{m}(s)\right|_{V} \diff s  \right] \\
\leq & C_{1}C_{2}\mathbb{E} \left[\int_{0}^{T} \vartheta^{\frac{1}{2}} \left|u^{m}(s)\right|_{V} \diff s  \right] \leq C_{1}C_{2} \mathbb{E} \left[\int_{0}^{T} \vartheta^{\frac{1}{2}} \tilde{C}_{2} \diff s\right] \leq C_{1}C_{2}\tilde(C)_{2}\vartheta^{\frac{1}{2}} = C_{3}^{1}\vartheta^{\frac{1}{2}},
\end{aligned}
\end{equation}
then $u_{3x}^{m}(t)$ fulfils assumptions of Lemma \ref{Aldous} for $a:=1$ and $b:=\frac{1}{2}$ with the norm $\left|\cdot\right|_{U'}$. 

Similarly
\begin{equation}
\begin{aligned}
\mathbb{E} & \left[ \left|u^{m}(\tau_{m}+\vartheta)u^{m}_{x}(\tau_{m}+\vartheta) - u^{m}(\tau_{m}+\vartheta)u^{m}_{x}(\tau_{m}+\vartheta)\right|_{U'} \right] =  \mathbb{E} \left[\left|\int_{\tau_{m}}^{\tau_{m}+\vartheta} u^{m}(s)u^{m}_{x}(s) \diff s \right|_{U'} \right]\\
\leq & \mathbb{E} \left[\left|\int_{\tau_{m}}^{\tau_{m}+\vartheta} \!\! u^{m}(s)u^{m}_{x}(s) \diff s \right|_{V'} \right] \leq \mathbb{E} \left[\int_{\tau_{m}}^{\tau_{m}+\vartheta} \!\! \left| u^{m}(s)u^{m}_{x}(s)\right|_{V'} \diff s  \right] \leq \frac{1}{2}C_{4} \mathbb{E} \left[\int_{\tau_{m}}^{\tau_{m}+\vartheta} \!\! \left|\left(u^{m}(s)\right)^{2} \right|_{H} \! \diff s \right] \\
\leq & \frac{1}{2}C_{4}C_{5}C_{6} \mathbb{E} \left[\int_{\tau_{m}}^{\tau_{m}+\vartheta} \left|\left(u^{m}(s)\right) \right|^{\frac{3}{2}}_{V}\left|\left(u^{m}(s)\right) \right|^{\frac{1}{2}}_{V} \diff s \right] \leq \frac{1}{2}C_{4}C_{5}C_{6} \mathbb{E} \left[\int_{0}^{T} \vartheta^{\frac{1}{2}}\left|\left(u^{m}(s)\right) \right|^{\frac{3}{2}}_{V}\left|\left(u^{m}(s)\right) \right|^{\frac{1}{2}}_{V} \diff s \right]   \\
\leq &  \frac{1}{2}C_{4}C_{5}C_{6}\mathbb{E} \left[\int_{0}^{T} \vartheta^{\frac{1}{2}}\tilde{C}_{2} \diff s\right] = \frac{1}{2}C_{4}C_{5}C_{6} \vartheta^{\frac{1}{2}}\tilde{C}_{2}T \leq C_{7}^{1}\vartheta^{\frac{1}{2}}.
\end{aligned}
\end{equation}
Therefore $u^{m}(t)u_{x}^{m}(t)$ fulfils assumptions of Lemma \ref{Aldous} for $a:=1$ and $b:=\frac{1}{2}$ with the norm $\left|\cdot\right|_{U'}$. 
In the case of all other terms the result from\cite{Mot} is used.
\begin{lemma}(\cite[p.~23]{Mot})
\begin{itemize}
\item[(i)] Let $F:[0,T]\times H\times Y\rightarrow H$ fufils conditions (F1)-(F4). Then the process $P_{m}F(s,u^{m};y)$ fulfils assumptions of Lemma  \ref{Aldous} for $a:=2$ and $b:=1$ with the norm $\left|\cdot\right|_{U'}$.
\item[(ii)] Let $\Phi:[0,T]\times V\rightarrow L_{2}^{0}(L^{2}(\mathbb{R})$ fufils conditions  (P1)-(P4). Then the process $P_{m}\Phi(s,u^{m}(s))$ fulfils assumptions of Lemma \ref{Aldous} for $a:=1$ and $b:=1$  with the norm  $\left|\cdot\right|_{U'}$.
\end{itemize}
\end{lemma}
Then due to  Lemma \ref{Aldous} the sequence $\left\{u^{m}(t)\right\}$ fulfils the Aldous condition in the space  $U'$, what finishes the proof.
\end{proof}

\appendix
\section*{Appendix A: Compensated time homogeneous Poisson random measure}\renewcommand{\thesection}{A} 

The following definition is cited from \cite{Mot}  (see also \cite{PeZa}).

Let $(\Omega,\mathcal{F},\mathbb{P}) $ be a complete probability space with filtration $\mathbb{F}:=(\mathcal{F})_{t\ge 0}$. 

\begin{definition} \label{a1}
Let $(Y,\mathcal{Y})$ be a measurable space. A {\tt  time homogeneous Poisson random measure} $\eta$ on $(Y,\mathcal{Y})$ over  $(\Omega,\mathcal{F},\mathbb{F},\mathbb{P})$ is a measurable function such that 
\begin{itemize}
\item[(i)] for all $B\in \mathcal{B}(\mathbb{R}_{+}) \otimes \mathcal{Y}, \eta(B) :=i_B \circ \eta : \Omega \to \bar{\mathbb{N}}$ is a Poisson  random measure with parameter $\mathbb{E}[\eta(B)]$;
\item[(ii)] $\eta$ is independently scattered, i.e.\ if the sets $B_j\subset \mathcal{B}(\mathbb{R}_{+}) \otimes \mathcal{Y}$, $j=1,\ldots ,n$, are disjoint then the random variables $\eta(B_j)$, $j=1,\ldots ,n$, are independent.
\item[(iii)] For all $U\in \mathcal{Y}$ the $\bar{\mathbb{N}}$-valued process $(N(t,U)_{t\ge 0}$ defined by 
$$ N(t,U):= \eta((0,t]\times U), \quad t\ge 0 $$
is $\mathbb{F}$-adapted and its increments are independent of the past, i.e.\ if $t>s\ge 0$, then $N(t,U)-N(s,U) = \eta((s,t]\times U)$ is independent on $\mathcal{F}_{s}$.
\end{itemize}

If $\eta$ is a time homogeneous Poisson random measure then the formula 
$$ \nu(A) := \mathbb{E} [\eta(0,1]\times A)], \quad A\in \mathcal{Y}  $$
defines a measure on $(Y,\mathcal{Y})$ called an {\tt intensity measure} of $\eta$. Moreover, for all $T<\infty$ and all $A\in \mathcal{Y}$ such that $\mathbb{E} [\eta(0,1]\times A)]<\infty$, the $\mathbb{R}$-valued process ${\tilde{N}(t,A)}_{t\in (0,T]}$ defined by 
$$ \tilde{N}(t,A) := \eta((0,T]\times A)- t\,\nu(A), \quad t\in (0,T] $$
is an integrable martingale on $(\Omega,\mathcal{F},\mathbb{F},\mathbb{P})$. The random measure $l \otimes\nu$ on $\mathcal{B}(\mathbb{R}_{+})\otimes \mathcal{Y}$, where $l$ stands for the Lebesgue measure, is called an {\tt compensator} of $\eta$ and the difference between a time homogeneous Poisson random measure $\eta$ and its compensator, i.e.
$$ \tilde{\eta} := \eta- l\otimes \nu,  $$
is called a {\tt compensated  time homogeneous Poisson random measure}.

\end{definition}

\end{document}